\documentclass[reqno]{amsart}
\usepackage{verbatim,amssymb,amsmath,amscd,latexsym,amsbsy,mathrsfs} 
\usepackage{graphicx}
\usepackage{relsize}
\usepackage[all,cmtip]{xy}
\usepackage[normalem]{ulem}
\usepackage[utf8]{inputenc}
\usepackage[T1]{fontenc}
\usepackage{lmodern}
\usepackage{lmodern}
\usepackage{lipsum}
\usepackage{macro}

\usepackage[all]{xy}
 \input{xy}
 \xyoption{all}
\usepackage{color}
\usepackage{xparse}
\usepackage{mathtools}
\usepackage[colorlinks = true,
            linkcolor = blue,
            urlcolor  = green,
            citecolor = green,
            anchorcolor = blue]{hyperref}   

\newtheorem{thm}{Theorem}[section]
\newtheorem{cor}[thm]{Corollary}

\newtheorem{prop}[thm]{Proposition}

\newtheorem{lemma}[thm]{Lemma}
\newtheorem{rmk}[thm]{Remark}

\newtheorem{note}[thm]{Notation}
\newtheorem{defin}[thm]{Definition}

\DeclareMathOperator*{\coker}{coker}

\DeclareMathOperator*{\Hom}{Hom}
\DeclareMathOperator*{\Ext}{Ext}

\newcommand{\RN}[1]{%
  \textup{\uppercase\expandafter{\romannumeral#1}}%
}

\newcommand {\C} {{\mathbb C}}

\newcommand {\Q} {{\mathbb Q}}

\begin{document}
\title[1-Motives Associated to the Limit MHS of Degenerations of Curves]{1-Motives Associated to the Limit Mixed Hodge Structures of Degenerations of Curves}
 
  \author{Feng Hao}

\address{Department of Mathematics\\
  Purdue University\\
  150 N. University Street\\
  West Lafayette, IN 47907\\
  U.S.A.\\fhao@purdue.edu.}

\begin{abstract} 
In this article, we will give the Deligne 1-motives up to isogeny corresponding to the $\mathbb{Q}$-limit mixed Hodge structures of semi-stable degenerations of curves, by using logarithmic structures and Steenbrink's  cohomological mixed Hodge complexes associated to semi-stable degenerations of curves.
\end{abstract}
\maketitle

\begin{center}
\section*{Introduction}
\end{center}

For a semi-stable degeneration of smooth projective curves, we have a limit mixed Hodge structure on the nearby fiber. The goal of this paper is to give a geometric construction of the 1-motive associated to the limit mixed Hodge structure.

Given a semi-stable degeneration of smooth projective curves $f: X\rightarrow \Delta$ over a small disk $\Delta$, with a chosen parameter $t:\Delta\rightarrow \mathbb{C}$. The central fiber of $f$, $X_0=f^{-1}(0)$, is a stable curve. Let $X^*=X-X_0$. Denote the inclusion $X_0\hookrightarrow X$ by $i$ and the inclusion $X^*\hookrightarrow X$ by $j$.
Then we have the following diagram
\[\tag{0.1}
\xymatrix{
X_{\infty}\ar@{}[dr] |(.3){\ulcorner}\ar[d]^-{f_{\infty}}\ar[r]^-e  &X^{*}\ar@{}[dr] |(.3){\ulcorner}\ar[d]^{f'}\ar@{^{(}->}[r]^j &X\ar[d]^f&X_0\ar@{}[dl] |(.3){\urcorner}\ar[d]\ar@{_{(}->}[l]_i \\
\mathfrak{h}\ar[r]^{\pi} &\Delta^*\ar@{^{(}->}[r] &\Delta &\{0\}\ar@{_{(}->}[l],\\
}
\]
where $\mathfrak{h}$ is the upper half plane, $\pi: \mathfrak{h}\rightarrow \Delta^*$ is the universal cover, and the squares are Cartesian. Denote the composition map $j\circ e$ by $k$. 

We call the fiber $X_{\infty}$ a \textit{nearby fiber}. It is homotopic to any fibre $X_s$ of $f$, for $s\neq0$. $X$ is homotopy equivalent to $X_0$ by a retraction $r: X\rightarrow X_0$.  We call the composition map $sp: X_s\overset{j}{\hookrightarrow} X\overset{r}{\rightarrow} X_0$ a \textit{specialization}.  \ Following from \cite{pest} and \cite{sga7}, we have the construction of a \textit{complex of nearby cocycles} $\Psi_f\underline{\mathbb{C}}_X:=i^*Rk_*k^*\underline{\mathbb{C}}_X$ (respectively, $\Psi_f\underline{\mathbb{Q}}_X:=i^*Rk_*k^*\underline{\mathbb{Q}}_X$) in the derived category $D^{+}(X_0, \mathbb{C})$ (respectively, $D^{+}(X_0, \mathbb{Q})$). According to Steenbrink \cite[lemma 4.3]{st1}, we have $$H^1(X_{\infty}, \mathbb{C})\cong \mathbb{H}^1(\Psi_f\underline{\mathbb{C}}_X)$$ (respectively, $H^1(X_{\infty}, \mathbb{Q})\cong \mathbb{H}^1(\Psi_f\underline{\mathbb{Q}}_X)$). Schmid \cite{sch} and Steenbrink \cite{st1} showed that the cohomology $H^1(X_{\infty}, \mathbb{C})$ admits a mixed Hodge structure, which is called the \textit{limit mixed Hodge structure}, as the ``limit'' of pure Hodge structures $H^1(X_s,\mathbb{C})$ of general fibers when $s$ approaches to $0$ in the disk $\Delta$.

For the construction of the limit mixed Hodge structure, note first that we have a de Rham complex over $X_0$ with logarithmic poles: 
$\Omega^{\bullet}_{X/\Delta}(\log X_0)\otimes_{\mathcal{O}_X}\mathcal{O}_{X_0}$, which is isomorphic to $\Psi_f\underline{\mathbb{C}}_X$ in the derived category $D^+(X_0, \mathbb{C})$. Steenbrink observed that the weight filtration of the limit mixed Hodge structure cannot be constructed on $\Omega^{\bullet}_{X/\Delta}(\log X_0)\otimes_{\mathcal{O}_X}\mathcal{O}_{X_0}$. To define the weight filtration of the limit mixed Hodge structure, Steenbrink constructed the so called \textit{Steenbrink double complex} $A_t^{{\bullet},{\bullet}}$, which  depends on the parameter $t: \Delta\rightarrow \mathbb{C}$ we chose.\ The Steenbrink double complex gives a resolution of $\Omega^{\bullet}_{X/\Delta}(\log X_0)\otimes_{\mathcal{O}_X}\mathcal{O}_{X_0}$, therefore there is an isomorphism $$\mathbb{H}^1(X_0, \Omega^{{\bullet}}_{X/\Delta}(\log X_0)\otimes_{\mathcal{O}_X}\mathcal{O}_{X_0})\cong \mathbb{H}^1(Tot(A_t^{{\bullet},{\bullet}})),$$ which depends on the choice of $t$. Over the Steenbrink double complex, the weight filtration and Hodge filtration can be defined, which give us the limit mixed Hodge structure after taking the hypercohomology. We denote the limit mixed Hodge structure by $H^1_{lim}(f,t)$ for the family $f$ and parameter $t$. We will recall the Steenbrink double complex in section \ref{double complex}.

On the other hand, Deligne \cite{de1} has built up an equivalence of two categories
\[
\left\{
\begin{array}{ll}
\text{\textit{$\mathbb{Z}$-mixed Hodge structure}}\\ (H_{\mathbb{Z}}, W_{\bullet}, F^{\bullet})
 \text{ of type}\ 
\{(-1,-1),\\
(-1,0), (0,-1), (0,0)\}\ \   \text{with}\\  
Gr^W_{-1}H\ \text{a polarized pure Hodge}\\ \text{structure of weight} -1.
\end{array}
\right\}
\Longleftrightarrow
\left\{
\begin{array}{ll}
\text{\textit{1-motive}}\ [L\overset{\mu}{\rightarrow} G], \ \text{with}\  L\ \text{a}\\ \text{free abelian group of finite rank},\\ \text{and}\  0\rightarrow
T\rightarrow G\rightarrow A\rightarrow 0\   
\text{an}\\ \text{extension of an abelian variety}\\ $A$\  \text{by a torus}\  T.
\end{array}
\right\}
 \]

We will briefly recall the above equivalence of categories in section \ref{abstract 1-motive}.  

For elementary geometric examples of the above equivalence of categories, we consider a smooth curve $X$ with $\overline{X}$ the smooth projective curve containing $X$ as an open subset. $H^1(X, \mathbb{C})(1)$ actually carries a $\mathbb{Z}$-mixed Hodge structure of type $\{(-1,0),(0,-1),(0,0)\}$, where ``$(1)$'' is the Tate twist. In the corresponding 1-motive $[L\rightarrow J(\overline{X})]$, $J(\overline{X})$ is the Jacobian variety of $\overline{X}$. $L$ is the free abelian group of rank $n-1$ with generators $\{p_1-p_0 ,..., p_{n-1}-p_0\}$, where the set of points $\{p_0, p_1,...,p_{n-1}\}$ is $\overline{X}-X$. The 1-motive map $\mu$ is actually the Abel-Jacobi map. For the algebraic construction of the 1-motive associated to a general curve, refer to Deligne \cite[section 10.3]{de1}. For the algebraic construction of 1-motives associated to the first cohomology of surfaces, refer to Carlson \cite{carl1}.

Go back to the limit mixed Hodge structure of a semi-stable degeneration of curves $f: X\rightarrow \Delta$. Fix a parameter $t: \Delta \rightarrow \mathbb{C}$.  
Steenbrink \cite{st1} constructed an abstract bifiltered $\mathbb{Q}$-cohomological mixed Hodge complex $((A_{t,\mathbb{Q}}^{\bullet}, W_{\bullet}), (A_t^{\bullet}, W_{\bullet}, F^{\bullet}))$, depending on $t$, which gives a $\mathbb{Q}$-mixed Hodge structure over $\mathbb{H}^1(Tot(A_t^{{\bullet},{\bullet}}))$. Later on, Steenbrink \cite{st2} constructed the $\mathbb{Q}$-mixed Hodge structure over $\mathbb{H}^1(Tot(A_t^{{\bullet},{\bullet}}))$ using log geometry. By using the monodromy weight spectral sequence, it is not hard to show that the $\mathbb{Q}$-limit mixed hodge structure $\mathbb{H}^1(Tot(A_t^{{\bullet},{\bullet}}))(1)$ is of type $\{(-1,-1),(-1,0),(0,-1),(0,0)\}$ (\cite[proposition 4.21]{st1}). Here we also denote it to be $H^1_{lim}(f,t)$. Abstractly, $H^1_{lim}(f,t)$ corresponds to a Deligne 1-motive $[L\overset{\mu_t}{\rightarrow} G]$ up to isogeny, by the correspondence of the two categories described above. Actually $G$ is $Pic^0(X_0)$ and $L$ is given by nodal points in $X_0$, which will be shown in section \ref{abs}. The main purpose of this article is to give an explicit geometric description of the 1-motive map $\mu_t$, which will be discussed in section \ref{geo 1}, \ref{geo 2}, and \ref{geo 3}. For the purpose of calculation, we will modify the construction of $\mathbb{Q}$-cohomological mixed Hodge complex constructed by Steenbrink through log geometry in Steenbrink \cite[section 5]{st2} and \cite[section 11.2.6]{pest}, to give a new $\mathbb{Q}$-cohomological mixed Hodge complex. We will use this new $\mathbb{Q}$-cohomological mixed Hodge complex to compute the 1-motive. Our main result is the following theorem.

\begin{thm}
The 1-motive corresponding to the $\Q$-limit mixed Hodge structure $H^1_{lim}(f,t)$ of the family $f$ and the parameter $t:\Delta\rightarrow \mathbb{C}$ is isogeny to the following homomorphism:
$$\nu_t: L\rightarrow Pic^0(X_0),$$ where

(1) $L\cong\ker\{H_0(X_0[2], \mathbb{Z})\rightarrow H_0(X_0[1], \mathbb{Z})\}$, where $X_0[1]$ is the normalization of $X_0$, and $X_0[2]$ is the set of nodal points of $X_0$;
    
(2) For $D=\underset{p\in X_0[2]}{\Sigma}n_p p\in L$,  the image $\nu_{t}(D)$ is the line bundle glued from the line bundle $$\mathcal{O}_{X_0[1]}(\underset{p\in X_0[2]}{\Sigma}n_p (p'-p''))$$ where $\{p', p''\}$ is the preimage of $p$ in $X_0[1]$, together with the following gluing data along each node:

For any node $p\in X_0$, choose

(1) an open analytic neighborhood $V_p\subset X$ of $p$;
    
(2) local coordinates $u, v$ over $V_p$, such that $uv=t\circ f$. Note that $\{u=0\}$ and $\{v=0\}$ are the two components of $V_p\cap X_{0}$. We regard $u|_{\{v=0\}}$ and $v|_{\{u=0\}}$ as local coordinate functions at $p'$ and $p''$.

Then the gluing isomorphism at each node $p$  is given by the following diagram
\[\xymatrixrowsep{0.3 in}
\xymatrixcolsep{0.2 in}
\xymatrix{
\mathcal{L}(D)\otimes(\mathcal{O}_{X_0[1],p'}/m_{p'})\ar[d]_-{u^{n_p}}^{\cong}\ar[r]&\mathcal{L}(D)\otimes(\mathcal{O}_{X_0[1],p''}/m_{p''})\ar[d]^-{(1/v)^{n_p}}_{\cong}\\
\mathcal{O}_{X_0[1],p'}/m_{p'}\ar[d]^{\cong}&\mathcal{O}_{X_0[1],p''}/m_{p''}\ar[d]^{\cong}\\
\mathbb{C}\ar@{=}[r]&\mathbb{C}.\\
}
\]
\end{thm}

We would like to mention the work by Jerome William Hoffman \cite{hoff}. He gives the 1-motives associated to $\mathbb{Z}$-limit mixed Hodge structures for semi-stable degenerations of curves, which coincides with the 1-motive in this paper after tensoring with $\mathbb{Q}$. The method in \cite{hoff} is more transcendental. Also, as an application, the behavior at infinity of the Torelli map is describle by the above 1-motive in \cite{hoff}.\\ 

\textbf{Acknowledgements.} I would like to thank my advisor Professor Donu Arapura for very helpful guidance and conversations throughout this project, and Partha Solapurkar for many useful discussions. I would also like to thank Professor Pierre Deligne for his priceless comments and suggestions.\\

\section{Logarithmic Structure and Koszul Complex} \label{LSKC}

In this section, we will briefly recall the logarithmic structures associated to the geometric setup in the introduction and some terminologies that will be used in the following sections.

\subsection{Logarithmic structure associated to semi-stable degenerations of curves}

For the semi-stable degeneration of curves $f: X\rightarrow \Delta$ given in the introduction, we have the associated \textit{log structure} for the embedding $i: X_0\hookrightarrow X$, which is the sheaf of monoids $\mathcal{M}_X:=\mathcal{O}_X\cap j_*\mathcal{O}_{X^*}^*$ together with the natural structure morphism $\alpha: \mathcal{M}_X\rightarrow\mathcal{O}_X$ such that $\alpha^{-1}(\mathcal{O}^*_X)\cong \mathcal{O}^*_X$. Similarly we have the log structure $\mathcal{M}_{\Delta}$ for the embedding $0\hookrightarrow \Delta$. An analytic space with a log structure $(X, \mathcal{M}_X)$ is called a \textit{log space}. The family $f$ gives a morphism of log spaces $\underline{f}: (X, \mathcal{M}_X)\rightarrow (\Delta, \mathcal{M}_{\Delta})$. We will omit the definition of morphisms of log spaces, since we will not use it in the following sections. For the definitions of a morphism of log spaces, refer to Steenbrink \cite[section 4]{st2} or Illusie \cite[section 1]{ill2}.

Restricting the log space $(X, \mathcal{M}_{X})$ to the central fiber $X_0$, we get a log space $(X_0, \mathcal{M}_{X_0})$, where $\mathcal{M}_{X_0}$ and the structure morphism $\beta:\mathcal{M}_{X_0}\rightarrow \mathcal{O}_{X_0}$ are defined through the following pushout square in the category of sheaves of monoids:
\[
\xymatrix{
&\gamma^{-1}(\mathcal{O}_{X_0}^*)\ar@{}[dr] |(.7){\lrcorner} \ar[d]^{\gamma} \ar@{^{(}->}[r] & i^{-1}\mathcal{M}_X\ar@/^/[ddr]^{\gamma} \ar[d] \\
&\mathcal{O}_{X_0}^* \ar@/_/[drr] \ar@{^{(}->}[r] &\mathcal{M}_{X_0}\ar[dr]^\beta\\
&&&\mathcal{O}_{X_0}
}
\tag{1.1}
\]
where $\gamma: i^{-1}\mathcal{M}_X\rightarrow \mathcal{O}_{X_0}$ is the composition map $$i^{-1}\mathcal{M}_X\rightarrow i^{-1}\mathcal{O}_X\rightarrow\mathcal{O}_{X_0}.$$ Hence $\mathcal{M}_{X_0}=\mathcal{O}_{X_0}^*\oplus i^{-1}\mathcal{M}_X/\sim$, where the equivalent relation is: $(f,m)\sim(f',m')$ if and only if there exists  $a,b \in \gamma^{-1}(\mathcal{O}_{X_0}^*)$ such that $m/m'=b/a$ and $f/f'=\gamma(a)/\gamma(b)$. 

For the log space $(X, \mathcal{M}_X)$, there exists a short exact sequence
\[
\xymatrix{0\ar[r] &\mathcal{O}_{X}^*\ar[r] &\mathcal{M}_{X}\ar[r]&(a_1)_*\mathbb{N}_{X_0[1]}\ar[r] &0},
\] where $a_1: X_0[1]\rightarrow X_0$ is the normalization of the nodal curve $X_0$. By abuse of notation, we also denote $a_1: X_0[1]\rightarrow X$ for the composition of the normalization morphism and the inclusion.

When we restrict the log space $(X, \mathcal{M}_X)$ to $X_0$, we still have the following short exact sequence \[
\xymatrix{0\ar[r] &\mathcal{O}_{X_0}^*\ar[r] &\mathcal{M}_{X_0}\ar[r]&(a_1)_*\mathbb{N}_{X_0[1]}\ar[r] &0}.
\]

Since the family $f$ is semi-stable, there exists a global section $\tilde{t}=t\circ f\in \Gamma(\mathcal{M}_{X})$ being mapped to $(1,...,1)$, the generator of $(a_1)_*\mathbb{N}_{X_0[1]}$. Similarly, there is a global section $\tilde{t}\in \Gamma(\mathcal{M}_{X_0})$ which is mapped to $(1,...,1)$.

Let's end this subsection with the following notations that we will use later.

\begin{note}
(1) Denote $X_0[2]$ the 2-fold intersection of the irreducible components of $X_0$, i.e., the set of nodal points of $X_0$, and $a_2: X_0[2]\rightarrow X$ be the natural inclusion map.

(2) Denote the number of the nodal points of $X_0$ by $d$ and the number of the irreducible components by $n$. \ For convenience, we fix an order of the irreducible components $\{X_{0,i}\}_{i=1}^{n}$ of $X_0$.
\end{note}


\subsection{Koszul complexes for homomorphisms of free abelian groups}

In this subsection we will list several constructions and results about the divided power envelop and the Koszul complex in Steenbrink \cite{st2}, Fujisawa \cite{fu2}, and Illusie \cite{ill1}. Our discussion is for free abelian groups, but the constructions and results can be generalized to sheaves of free abelian groups.

Let $E$ be a free abelian group. Denote $T(E)$, $S(E)$, and $\bigwedge(F)$ to be the tensor algebra, the symmetric algebra, and the exterior algebra of $E$, respectively. For $x\in E$ and $n\in \mathbb{N}$, denote $\gamma_n(x)=x^n/n!$. Note that $\gamma_n(x+y)=\Sigma_{i+j=n}\gamma_i(x)\gamma_j(y)$. Let $\{e_{\alpha}\}$ be a basis of $E$. Define $\Gamma_n(E)$ to be a free subgroup of $S_n(E\otimes\mathbb{Q})$ generated by $\{\gamma_{n_1}(e_{\alpha_1})...\gamma_{n_m}(e_{\alpha_m})| \Sigma n_i=n\}$, then $\Gamma(E)=\oplus_n\Gamma_n(E)\subset S(E_{\mathbb{Q}})$ is a subalgebra.

For any homomorphism of free abelian groups $\varepsilon: E\rightarrow F$ we have an induced homomorphism $\Gamma(\varepsilon)$: $\Gamma(E)\rightarrow \Gamma(F)$ of graded $\mathbb{Z}$-algebras. Then we have the bigraded algebra $Kos(\varepsilon):=\Gamma(E)\otimes\bigwedge(F)$. Note that $Kos(\varepsilon)$ is commutative in the first degree and anti-commutative in the second degree.

\begin{lemma}[{\cite[proposition 4.3.1.2]{ill1}}] \label{1.2.1}
For any homomorphism of finitely generated free abelian groups $\varepsilon: E\rightarrow F$, there exists a unique endomorphism $d$ of bidegree $(-1, 1)$ of $Kos(\varepsilon)$ satisfying the following axioms:

(1) $d(xx')=(dx)x'+(-1)^qx(dx')$ for $x$ a homogeneous element of bidegree $(p, q)$ in Kos$(\varepsilon)$;

(2) $d(\gamma_k(x)\otimes 1)=\gamma_{k-1}(x)\otimes \varepsilon(x)$ for $x\in E$;

(3) $d(1\otimes y)=0$ for $y\in F$.

\end{lemma}

Fix an integer $n$, denote $Kos^{n}(\varepsilon)^q:=\Gamma_{n-q}(E)\otimes \overset{q}{\bigwedge}(F)$. Then we have a cochain complex (with respect to index $q$) of free abelian groups $Kos{^n(\varepsilon)}^{\bullet}$, with the above differential $d$:
$$\cdots\longrightarrow Kos^{n}(\varepsilon)^{q-1}\stackrel{d^{q-1}}{\longrightarrow}Kos^{n}(\varepsilon)^q\stackrel{d^{q}}{\longrightarrow}Kos^{n}(\varepsilon)^{q+1}\stackrel{d^{q+1}}{\longrightarrow}\cdots$$
The above differential $d^{q}$ is explicitly given by:$$ d^{q}(\gamma_{n_1}(x_1)..\gamma_{n_k}(x_k)\otimes y)=\Sigma^{k}_{i=1}\gamma_{n_1}(x_1)..\gamma_{n_i-1}(x_i)..\gamma_{n_k}(x_k)\otimes \varepsilon(x_i)\wedge y,$$
where $x_i\in E,\  \Sigma_in_i=n-q,\  y\in\overset{q}{\bigwedge}(F)$.

\begin{lemma}[{\cite[proposition 4.3.1.6]{ill1}}]\label{1.2.2}
Suppose that  $\varepsilon$ is a homomorphism as that in lemma \ref{1.2.1}. Moreover, assume $\coker(\varepsilon)$ is a free abelian group. Then we have the following isomorphism:

$$H^p(Kos^n(\varepsilon)^{\bullet})\cong\Gamma_{n-p}(\ker(\varepsilon))\otimes\overset{p}{\bigwedge}(\coker(\varepsilon)),$$

where $H^p(Kos^n(\varepsilon)^{\bullet})$ is the $p$-th cohomology of the complex $Kos^n(\varepsilon)^{\bullet}$.
\end{lemma}

\begin{rmk}[{\cite[(1.5)]{fu1}}]\label{1.2.3}
For any free abelian subgroup $G\subset F$ and any integer $0\leq m\leq q$, we can define a subgroup $W(G)_mKos^{n}(\varepsilon)^q$ of $Kos^{n}(\varepsilon)^q$ to be the image of the morphism $$\Gamma_{n-q}(E)\otimes\overset{q-m}{\bigwedge}G\otimes\overset{m}{\bigwedge}F\longrightarrow\Gamma_{n-q}(E)\otimes\overset{q}{\bigwedge}F; \ x\otimes y\otimes z\mapsto x\otimes(y\wedge z).$$ Also, we define $W(G)_mKos^{n}(\varepsilon)^q=0$ for $m<0$, and $W(G)_mKos^{n}(\varepsilon)^q=Kos^{n}(\varepsilon)^q$ for $m>q$. Moreover, if $\varepsilon(E)\subset G$, the above construction gives us a finite increasing filtration $W(G)_{\bullet}$ on the complex $Kos^n(\varepsilon)^{\bullet}$.
\end{rmk}

\begin{prop}[{\cite[proposition 1.6]{fu1}}]\label{1.2.4}
Suppose that $E, G, F, and \ \varepsilon: E\rightarrow F$ are as that in remark \ref{1.2.3}. Denote $\varepsilon_G: E\rightarrow G$ to be the induced map by $\varepsilon$. Then we have the following isomorphism of complexes induced by the morphism in remark \ref{1.2.3}:

$$Gr^{W(G)}_mKos^n(\varepsilon)^{\bullet}\cong Kos^{n-m}(\varepsilon_G)^{\bullet}[-m]\otimes \overset{m}{\bigwedge}(F/G),$$
where $[-m]$ means shifting the complex by $m$.
\end{prop}

Note that in Fujisawa \cite{fu1}, the above proposition is stated for sheaves of free abelian groups.

\begin{rmk} \label{1.2.5}
These constructions about the Koszul complex above can be generalized to any flat $A$-module, where $A$ is a subalgebra of $\mathbb{C}$, e.g. $\mathbb{Q},$ and $\mathbb{C}$.
\end{rmk}

\section{Steenbrink's Double Complex}\label{double complex}

In this section, we will recall the Steenbrink's limit cohomological mixed Hodge complex on the complex level and construct a modified Steenbrink's limit cohomological mixed Hodge complex on the rational level, associated to the geometric setting described in the introduction.

\subsection{The complex structure}

Consider the sheaves of log differentials $\Omega^i_X(\log  X_0), i=0,1,2$. For each $i$, we have the \textit{Deligne's filtration}:
\begin{equation*}
    W_p{\Omega_X^i}(\log X_0):=
    \begin{cases}
  0 & \text{for}\  p<0,\\
  \Omega_{X}^i(\log X_0) & \text{for}\  p\geq i,  \\
  \Omega_{X}^{i-p}\wedge\Omega_{X}^{p}(\log X_0) & \text{for}\  0\leq p\leq i.
  \end{cases}
\end{equation*}

Fix a parameter $t: \Delta\rightarrow \mathbb{C}$. Define the \textit{Steenbrink double complex} as follows.

$A_t^{p,q}:=\Omega_{X}^{p+q+1}(\log X_0)/W_p\Omega_{X}^{p+q+1}(\log X_0)$ with two differentials:

(1) $d':A_t^{p,q}\rightarrow A_t^{p+1,q}$ by $d'(\omega):=$ class of $\theta\wedge\omega$, where $\theta=f^*(dt/t)$, $t$ is the parameter chosen for $\Delta$;

(2) $d'':A_t^{p,q}\rightarrow A_t^{p,q+1}$ by $d''(\omega):=$ class of $d\omega$.\

Define the monodromy weight filtration $M_{\bullet}$ and Hodge filtration $F^{\bullet}$ of the double complex as follows.

$M_rA_t^{p,q}:=$ The image of $W_{2p+r+1}\Omega_{X}^{p+q+1}(\log X_0)$ in $A_t^{p,q}$.

The Hodge filtration $F^{\bullet}$ is defined by the ``stupid filtration'' of the double complex.

Specifically, for the semi-stable degeneration $f$ of curves, we have the double complex $A_t^{\bullet,\bullet}$:
\[
\xymatrix{
&\frac{\Omega_{X}^2(\log X_0)}{\Omega_{X}^2}\\
&\frac{\Omega_{X}^{1}(\log X_0)}{\Omega_{X}^1}\ar[u]_{d''}\ar[r]^-{d'}  &\frac{\Omega_{X}^2(\log X_0)}{\Omega_{X}^1\wedge\Omega_{X}^1(\log X_0)}. \\
}
\tag{2.1}
\]

If we denote the total complex $Tot(A_t^{\bullet,\bullet})$ of the above double complex to be $A_t^{\bullet}$, then we have the $\mathbb{C}$-structure of a cohomological mixed Hodge complex $(A_t^{\bullet}, M_{\bullet}, F^{\bullet})$, which gives the $\C$-limit mixed Hodge structure associated to the semi-stable degeneration of curves.  From now on, for simplicity, we just denote $A_t^{\bullet}$ by $A^{\bullet}$, for the fixed parameter $t$.

\subsection{The rational structure}

In this section, we follow the terminologies and constructions in Steenbrink \cite{st2} to give a rational structure of the limit  mixed Hodge structures of a semistable degeneration of curves.

Under the basic constructions in section \ref{LSKC},  we consider the morphism $\textbf{e}: \mathcal{O}_{X}\rightarrow\mathcal{M}_{X}:=\mathcal{O}_X\cap j_*\mathcal{O}_{U}^*$, which is the composition of the exponential map $e^{2\pi \sqrt[]{-1}(\cdot)}: \mathcal{O}_{X}\rightarrow\mathcal{O}^*_{X}$ and the inclusion map $\mathcal{O}^*_{X}\hookrightarrow \mathcal{M}_{X}$.
Then we have the following exact sequence of sheaves (Steenbrink \cite[lemma 2.7]{st2}):
\[
\xymatrix{0\ar[r] &\mathbb{Z}_{X}\ar[r]&\mathcal{O}_{X}\ar[r]^{\textbf{e}} &\mathcal{M}^{gp}_{X}\ar[r]&(a_1)_*\mathbb{Z}_{X_0[1]}\ar[r] &0},
\]
where $\mathcal{M}^{gp}_{X}$ is the groupification of the sheaf of monoids $\mathcal{M}_{X}$.

As in section \ref{LSKC}, restricting the log structure $\mathcal{M}_{X}$ along the closed immersion $i: X_0\hookrightarrow X$ to $X_0$ provides a log structure $\mathcal{M}_{X_0}$ over $X_0$.  Then we have the following exact sequence over $X_0$ (Steenbrink \cite[(3.9)]{st2}):
\[
\xymatrix{0\ar[r] &\mathbb{Z}_{X_0}\ar[r]&\mathcal{O}_{X_0}\ar[r]^{\textbf{e}} &\mathcal{M}^{gp}_{X_0}\ar[r]&(a_1)_*\mathbb{Z}_{X_0[1]}\ar[r] &0}.
\]

After tensoring with $\mathbb{Q}$, we get an exact sequence
\[
\xymatrix{0\ar[r] &\mathbb{Q}_{X_0}\ar[r]&\mathcal{O}_{X_0}\ar[r]^-{\textbf{e}} &\mathcal{M}^{gp}_{X_0}\otimes \mathbb{Q}\ar[r]&(a_1)_*\mathbb{Q}_{X_0[1]}\ar[r] &0},
\]
with the global section $\tilde{t}\in \Gamma (X_0, \mathcal{M}_{X_0}^{gp}\otimes\mathbb{Q})$ which is mapped to (1,...,1).  From now on, for simplicity, we will denote $\mathcal{M}_{X_0}^{gp}\otimes\mathbb{Q}$ by $\mathcal{M}$.

Consider the morphism $\textbf{e}: \mathcal{O}_{X_0} \rightarrow \mathcal{M}$ of sheaves of $\Q$-vector spaces in the above exact sequence. We have the following data in terms of terminologies in section \ref{LSKC}:

(1) $K^q:=Kos^{2}(\textbf{e})^q=\Gamma_{2-q}(\mathcal{O}_{X_0})\otimes{\overset{q}{\bigwedge}}\mathcal{M}$;\ \ \ \ \ \ \ \ \ \ \ \ \ \ \ \ \ \ \ \ \ \ \ \ \ \ \ \ \ \ \ \ \ \ \ \ \ \ \ \ \ \ \ \ \ \ \ \ \ \  \ (2.2)\label{formula 2.2}

(2) $W_mK^q:=$The image of $\Gamma_{2-q}(\mathcal{O}_{X_0})\otimes{\overset{q-m}{\bigwedge}}\textbf{e}(\mathcal{O}_{X_0})\otimes{\overset{m}{\bigwedge}}\mathcal{M}$ in $K^q$,


\begin{thm}\label{2.2.1}
The morphism
$$\phi_m: W_mK^q\rightarrow W_m\Omega_X^q(\log X_0);$$ $$f_1\cdot ...\cdot f_{2-q}\otimes \textbf{e}(g_1)\wedge...\wedge \textbf{e}(g_{q-m})\otimes y_1\wedge...\wedge y_m$$ $$\mapsto (2\pi \sqrt[]{-1})^{-q}(\underset{i=1}{\overset{2-q}{\Pi}}f_i) dg_1\wedge...\wedge dg_{q-m}\wedge \frac{dy_1}{y_1}\wedge...\wedge\frac{dy_m}{y_m}$$
induces a filtered quasi-isomorphism between $(K^{\bullet}\otimes \mathbb{C}, W_{\bullet})$ and $(\Omega_X^{\bullet}(\log X_0), W_{\bullet}).$
\end{thm}

We will reproduce the proof here. The proof is essentially the same with the one in \cite[section (2.8)]{st2}.

\begin{proof}
It suffices to show that $Gr^W_m(\phi_{\bullet}): Gr_m^WK^{\bullet}\otimes \mathbb{C}\rightarrow Gr_m^W(\Omega_X^{\bullet}(\log X_0)$ is a quasi-isomorphism for each $m$.

(1) $m=0$

$\mathcal{H}^0(Gr^W_0\Omega_X^{\bullet}(\log X_0))\cong\mathbb{C}_{X_0}$ and \ $\mathcal{H}^i(Gr^W_0\Omega_X^{\bullet}(\log X_0))=0$, for $i>0$, by Steebrink  \cite[corollary 1.9]{st1}.

By proposition \ref{1.2.4}, we have $$Gr_0^WK^{\bullet}\cong Kos^{2}(\mathcal{O}_{X_0}\overset{e^{2\pi \sqrt[]{-1}(\cdot)}}{\longrightarrow}\mathcal{O}_{X_0}^*\otimes \mathbb{Q})^{\bullet}.$$

Then by lemma \ref{1.2.2}, we have
$\mathcal{H}^0(Gr_0^WK^{\bullet})\cong\mathbb{Q}_{X_0}$ and $\mathcal{H}^i(Gr_0^WK^{\bullet})=0$, for $i>0$.

(2) $m=1$

$\mathcal{H}^1(Gr^W_1\Omega_X^{\bullet}(\log X_0))\cong(a_1)_*\mathbb{C}_{X_0[1]}$ and \ $\mathcal{H}^i(Gr^W_1\Omega_X^{\bullet}(\log X_0))=0$, for $i\neq1$, by Steebrink  \cite[corollary 1.9]{st1}.

By proposition \ref{1.2.4}, we have $$Gr_1^WK^{\bullet}\cong Kos^{1}(\mathcal{O}_{X_0}\overset{e^{2\pi \sqrt[]{-1}(\cdot)}}{\longrightarrow}\mathcal{O}_{X_0}^*\otimes \mathbb{Q})^{\bullet}[-1]\otimes \overset{1}{\bigwedge}(a_1)_*\mathbb{Q}_{X_0[1]}.$$

Then by lemma \ref{1.2.2}, we have
$\mathcal{H}^1(Gr_1^WK^{\bullet})\cong(a_1)_*\mathbb{Q}_{X_0[1]}$ and $\mathcal{H}^i(Gr_1^WK^{\bullet})=0$, for $i\neq 1$.

The argument for $m\geq 2$ is exactly the same.
\end{proof}

Let's construct Steenbrink's double complex in Steenbrink \cite{st2} over $\Q$. The original double complex in Steenbrink \cite{st2} is constructed over $\mathbb{Z}$ by using log geometry. Here we reconstruct it over $\mathbb{Q}$ by using the data (\ref{formula 2.2}) and the above results. Since after tensoring with $\Q$, the torsion in $\mathcal{M}_{X_0}^{gp}$ is already killed, so we don't need to do the operation as Steenbrink did in \cite[section 2.8 and 5.2]{st2}.

\begin{note}\label{2.2.2}
Let $K^{\bullet}$ be a complex. Denote $K^p(n)[m]:=(2\pi\sqrt[]{-1})^nK^{p+m}$. 
\end{note}

\noindent\textbf{Construction}\ ($\mathbb{Q}$-Steenbrink double complex).\ For the fixed parameter $t$, let $A_{t,\mathbb{Q}}^{p,q}=(K^{p+q+1}/W_pK^{p+q+1})(p+1)$ for $p,q\ge 0$ with differentials

(1) $d':A_{t,\mathbb{Q}}^{p,q}\rightarrow A_{t,\mathbb{Q}}^{p+1,q}$ 
induced by cup product with $2\pi\sqrt[]{-1}\tilde{t}$ 

(2) $d'':A_{t,\mathbb{Q}}^{p,q}\rightarrow A_{t,\mathbb{Q}}^{p,q+1}$
the differential $d$ of the kozul complex.

We also have the monodromy weight filtration $M_{\bullet}$ 
$$M_rA_{t,\mathbb{Q}}^{p,q}:= \text{the image of} \ \ W_{2p+r+1}K^{p+q+1}(p+1)\  \text{in} \ A_{t,\mathbb{Q}}^{p,q}.$$

Denote the total complex $Tot(A_{t,\mathbb{Q}}^{\bullet, \bullet})$ by $A_{t,\mathbb{Q}}^{\bullet}.$ From now on, for simplicity, we just denote $A_{t,\mathbb{Q}}^{\bullet}$ by $A_{\mathbb{Q}}^{\bullet}$, for the fixed parameter $t$. 

The morphism in theorem \ref{2.2.1} 
$$\phi_m: W_mK^q\rightarrow W_m\Omega_X^q(\log X_0)$$
induces a morphism of filtered double complexes 
$$\phi:(A_{\mathbb{Q}}^{\bullet, \bullet}, M_{\bullet})\rightarrow (A^{\bullet, \bullet}, M_{\bullet}).$$

\begin{thm}\label{2.2.3}
(1) $Gr^M_mA_{\mathbb{Q}}^{\bullet}\cong \underset{p\geq 0,-m}{\bigoplus}(Gr^W_{m+2p+1}K^{\bullet})(p+1)[1]$.

(2) $\phi: (A_{\mathbb{Q}}^{ \bullet}\otimes\mathbb{C}, M_{\bullet})\rightarrow(A^{\bullet}, M_{\bullet})$\textit{ is a filtered quasi-isomorphism.}

(3)\ \textit{We have the following commutative diagram}

\[\xymatrix{
&i^*Rk_*k^*K^{\bullet}\otimes\mathbb{C}\ar[d]&i^*K^{\bullet}\otimes\mathbb{C}\ar[l]\ar[d]\ar[r] &A_{\mathbb{Q}}^{\bullet}\otimes\mathbb{C}\ar[d] \\
(Nearby Cycle)\ar@{~>}[r]&i^*Rk_*k^*\mathbb{C}_X&i^*\Omega_X^{\bullet}(\log X_0)[\log t]\ar[l]\ar[r] &A^{\bullet},\\
}
\tag{2.3}\label{diag 2.3}
\] such that every morphism in the above diagram is a quasi-isomorphism. This means that the $\mathbb{Q}$-cohomological mixed Hodge complex $(A_{\mathbb{Q}}^{\bullet}, M_{\bullet}), (A^{\bullet}, M_{\bullet}, F^{\bullet})$ coincides with the one in Steenbrink \cite[section 4]{st1}.
\end{thm}

\begin{proof} (1) follows from direct computations. 

$M_rA_{\mathbb{Q}}^{\bullet,\bullet}=$ 
\[
\xymatrix{
&0&0 \\
&\frac{W_{r+1}K^2}{W_0K^2}(1)\ar[u]\ar[r]&\frac{W_{r+3}K^3}{W_1K^3}(2)\ar[u]\ar[r]&\cdots\\
(0,0)\ar@{~>}[r] &\frac{W_{r+1}K^1}{W_0K^1}(1)\ar[u]\ar[r]&\frac{W_{r+3}K^2}{W_1K^2}(2)\ar[u]\ar[r] &\cdots\\
&0\ar[u]&0\ar[u].\\
}
\tag{2.4}\label{diag 2.4}
\]
Thus $Gr_r^MA_{\mathbb{Q}}^{\bullet,\bullet}$ equals to the following double complex
\[
\xymatrix{
&0&0 \\
&\frac{W_{r+1}K^2}{W_rK^2}(1)\ar[u]\ar[r]&\frac{W_{r+3}K^3}{W_{r+2}K^3}(2)\ar[u]\ar[r]&\cdots\\
(0,0)\ar@{~>}[r] &\frac{W_{r+1}K^1}{W_rK^1}(1)\ar[u]\ar[r]&\frac{W_{r+3}K^2}{W_{r+2}K^2}(2)\ar[u]\ar[r] &\cdots\\
&0\ar[u]&0\ar[u].\\
}
\tag{2.5}\label{diag 2.5}
\]
Note that all the horizontal morphisms in diagram (\ref{diag 2.5}) are zero. Thus the total complex is a direct sum of vertical complexes with a shift by 1, i.e., $$Gr^M_mA_{\mathbb{Q}}^{\bullet}\cong \underset{p\geq 0,-m}{\bigoplus}(Gr^W_{m+2p+1}K^{\bullet})(p+1)[1].$$

(2) First, claim that $Gr_m^MA_{\mathbb{Q}}^{\bullet}$ is quasi-isomorphic to 0 for $m\neq -1,0,1$.

In fact, by theorem \ref{2.2.1}, $Gr_{m+2p+1}^WK^{\bullet}(p+1)[1]\otimes \mathbb{C}$ is quasi-isomorphic to the complex  $Gr_{m+2p+1}^W\Omega^{\bullet}_X(\log X_0)[1]$. Since $Gr_r^W\Omega^{\bullet}_X(\log X_0)=0$, for $r\neq 0,1,2$, $Gr_m^MA_{\mathbb{Q}}^{\bullet}$ is quasi-isomorphic to 0 for $m\neq -1,0,1$ by theorem \ref{2.2.3} (1).

For $r=-1$, $$Gr^M_{-1}A_{\mathbb{Q}}^{\bullet}\otimes\mathbb{C}\cong Gr^W_{2}K^{\bullet}(2)[1]\otimes\mathbb{C}\cong Gr^W_{2}\Omega^{\bullet}_X(\log X_0)[1]\cong Gr^M_{-1}A^{\bullet}.$$ The last isomorphism follows from Steenbrink \cite[lemma 4.18]{st1}. For $r=0,1$, the calculations are the same.

(3) 
Note that the bottom row of diagram (2.3) contains two quasi-isomorphisms constructed by N. Katz in Steenbrink \cite[section 2.6]{st1}. The last vertical arrow in diagram (\ref{diag 2.3}) is a quasi-isomorphism by theorem \ref{2.2.3} (2). The first and middle vertical arrows are quasi-isomorphisms by theorem \ref{2.2.1}, for $n=0$. Thus the top row contains two quasi-isomorphisms.
\end{proof}

\section{1-Motives Associated to Abstract Mixed Hodge Structures}\label{abstract 1-motive}

In this section, we will briefly recall the theory of Deligne 1-motives. Following from \cite[(10.1.3)]{de1}, we have an equivalence of categories as mentioned in the introduction:
\[
\left\{
\begin{array}{ll}
\text{\textit{$\mathbb{Z}$-mixed Hodge structure}}\\ (H_{\mathbb{Z}}, W_{\bullet}, F^{\bullet})
 \text{ of type}\ 
\{(-1,-1),\\
(-1,0), (0,-1), (0,0)\}\ \   \text{with}\\  
Gr^W_{-1}H\ \text{a polarized pure Hodge}\\ \text{structure of weight} -1.
\end{array}
\right\}
\Longleftrightarrow
\left\{
\begin{array}{ll}
\text{\textit{1-motive}}\ [L\overset{\mu}{\rightarrow} G], \ \text{with}\  L\ \text{a}\\ \text{free abelian group of finite rank},\\ \text{and}\  0\rightarrow
T\rightarrow G\rightarrow A\rightarrow 0\   
\text{an}\\ \text{extension of an abelian variety}\\ $A$\  \text{by a torus}\  T.
\end{array}
\right\}
 \]
  
(1) For an arbitrary 1-motive $M$
\[
\xymatrix{
&&&L\ar[d]^{\mu}\\
&0\ar[r]&T\ar[r]&G\ar[r]&A\ar[r]&0, \\
}
\]
the corresponding mixed Hodge structure is constructed as follows
\[
\xymatrix{
&0\ar[r]&H_1(G)\ar[r]&\text{Lie}G\ar[r]^{exp}&G\ar[r]&0\\
&0\ar[r]&H_1(G)\ar@{=}[u]\ar[r]&H_{\mathbb{Z}}\ar@{-->}[r]\ar@{-->}[u]^{\alpha}&L\ar[r]\ar[u]^{\mu}&0, \\
}
\]
where the square on the right is the pullback of morphisms $\mu$ and $exp$.

$H_{\mathbb{Z}}$ is the integral lattice of the corresponding mixed Hodge structure. Denote $H:=H_{\mathbb{Z}}\otimes\mathbb{C}$. \  The Weight filtrations are $W_{-1}H_{\mathbb{Z}}:=H_1(G)$, $W_{-2}H_{\mathbb{Z}}:=\ker\{H_1(G)\rightarrow H_1(A)\}$. The Hodge filtration is $F^0H:=\ker\{\alpha\otimes\mathbb{C}: H\rightarrow \text{Lie}G\}$.

(2) For the other direction, we start from a mixed Hodge structure $(H_{\mathbb{Z}}, W_{\bullet}, F^{\bullet})$ of the above type, to construct the corresponding 1-motive.

Firstly, we introduce a construction given by Deligne \cite{de1}.

Since $Gr_{-1}^WH$ is polarizable, the complex torus
$A=Gr^W_{-1}H_{\mathbb{C}}/(F^0Gr_{-1}^WH_{\mathbb{C}}+Gr_{-1}^WH_{\mathbb{Z}})$
is an abelian variety. Let $T$ be the torus of the character group of the dual of $Gr_{-2}^W(H_{\mathbb{Z}})$. 
The complex analytic group $G=W_{-1}H_{\mathbb{C}}/(F^0W_{-1}H_{\mathbb{C}}+W_{-1}H_{\mathbb{Z}})$
is an extension of $A$ by $T$, as a semi-abelian variety.  \ Let $L=Gr_0^WH_{\mathbb{Z}}$. Then the corresponding 1-motive is the map $\mu: L\rightarrow G$ making the following diagram commutative:
\[
\xymatrix{
&0\ar[r]&W_{-1}H_{\mathbb{Z}}\ar[r]&W_{-1}H_{\mathbb{C}}/F^0W_{-1}H_{\mathbb{C}}\ar[r]&G\ar[r]&0\\
&&&H_{\mathbb{C}}/F^0\ar[u]^{\cong}\\
&0\ar[r]&W_{-1}H_{\mathbb{Z}}\ar@{=}[uu]\ar[r]&H_{\mathbb{Z}}\ar[r]\ar[u]&L\ar[r]\ar[uu]^{\mu}&0.\\
}
\]

\begin{rmk}\label{3.0.1}
When we consider $\mathbb{Q}$-mixed Hodge structures, we get 1-motives up to isogeny, i.e., we only determine the morphism $\mu$ upto tensoring with $\Q$ in the above diagram.
\end{rmk}

Next we give another construction by J. Carlson, which will not be use it in the following sections.  Recall the following theorem about an extension of mixed Hodge structures.

\begin{thm}[{\cite[theorem 3.31]{pest}}]
Let $A$ and $B$ be $\mathbb{Q}$-mixed Hodge structures, then there is a canonical isomorphism
$${\Ext}_{MHS}^1(A, B)\cong W_0\Hom(A, B)_{\mathbb{C}}/(F^0W_0\Hom(A, B)_{\mathbb{C}}+W_0\Hom(A,B)_{\mathbb{Q}}).$$
\end{thm}

\begin{cor}\label{3.0.3}
Note that if $A$ and $B$ are separated mixed Hodge structures (i.e., the weights of B are less than the weights of A), we have
$${\Ext}_{MHS}^1(A, B)\cong \Hom(A, B)_{\mathbb{C}}/(F^0\Hom(A, B)_{\mathbb{C}}+\Hom(A,B)_{\mathbb{Q}}).$$
\end{cor}

In fact, the isomorphism is given explicitly as follows.

Let
\[0\rightarrow B_{\mathbb{Q}}\overset{\beta}{\rightarrow} H_{\mathbb{Q}}\overset{\alpha}{\rightarrow} A_{\mathbb{Q}}\rightarrow 0\]
be a separated extension of $\mathbb{Q}$-mixed Hodge structures. Choose any retraction $r: H_{\mathbb{Q}}\rightarrow B_{\mathbb{Q}},$ i.e., $r\circ\beta=id_B$ and a section $\sigma_F$ of $\alpha_{\mathbb{C}}: H_{\mathbb{C}}\rightarrow A_{\mathbb{C}}$ preserving Hodge filtration, then the above extension corresponds to the element represented by $r\otimes\mathbb{C}\circ \sigma_F$ in corollary \ref{3.0.3}.\ \  If $B$ is a mixed Hodge structure of type $\{(-1,-1),(-1,0),(0,-1)\}$ and $A$ is a mixed Hodge structure of type $\{(0,0)\}$, the 1-motive upto isogeny of the above extension $H$ is given by the morphism:
\[A_{\mathbb{Q}}\overset{\sigma_F|_{A_{\mathbb{Q}}}}{\longrightarrow} H_{\mathbb{C}}\overset{r\otimes{\mathbb{C}}}{\longrightarrow} B_{\mathbb{C}}/(F^0B_{\mathbb{C}}+B_{\mathbb{Q}})\] by Carlson \cite[proposition 3, lemma 4]{carl2}.

\section{1-Motives upto Isogeny Associated to Semi-Stable Degenerations of Curves}

\subsection{Abstract construction for 1-motives}\label{abs}

For the family $f:X\rightarrow\Delta$ and parameter $t$ as that in the introduction, the limit mixed Hodge structure $$((\mathbb{H}^1(A_{\mathbb{Q}}^{\bullet}), W_{\bullet}), (\mathbb{H}^1(A^{\bullet}), W_{\bullet}, F^{\bullet}))$$ constructed in section \ref{double complex} is  of type $\{(0,0),(0,1),(1,0), (1,1)\}$. After taking Tate twist, it is of type $\{(-1,-1),(0,-1)$, $(-1,0), (0,0)\}$. Also, by the following lemma \ref{4.1.2}, $Gr^W_{-1}\mathbb{H}^1(A^{\bullet})$ is polarized. Then we have an associated 1-motive $\mu_t: L\rightarrow G$ upto isogeny as explained in section \ref{abstract 1-motive}. In this subsection, we give an abstract description of the 1-motive map $\mu_t$ up to isogeny.\ From now on we will only write $\mathbb{H}^1(A^{\bullet})$ for the $\mathbb{Q}$-limit mixed Hodge structure and write $\mathbb{H}^1(A_{\mathbb{Q}}^{\bullet})$ for its underling $\mathbb{Q}$ structure.

\begin{lemma}\label{4.1.1}
$Gr_2^W\mathbb{H}^1(A_{\mathbb{Q}}^{\bullet})\cong \ker\{H^0(X_0[2], \mathbb{Q}(1))\rightarrow H^2(X_0[1], \mathbb{Q}(1))\}$, where $\mathbb{Q}(1)$ means $2\pi\sqrt[]{-1}\mathbb{Q}$.
\end{lemma}

\begin{proof} Consider the monodromy weight spectral sequence 
\begin{equation*}
E_1^{-m,1+m}=\mathbb{H}^1(Gr^{M}_mA_{\mathbb{Q}}^{\bullet})\Rightarrow \mathbb{H}^1(A_{\mathbb{Q}}^{\bullet})
\end{equation*}
which degenerates at $E_2$ page, i.e., \begin{equation*}
    \begin{aligned}
E_2^{-m,1+m}=&\text{Cohomology of} \ \{E_1^{-m-1,1+m}\rightarrow {E}_1^{-m,1+m}\rightarrow {E}_1^{-m+1,1+m}\}\\
=&{E}_{\infty}^{-m,1+m}\\
=&Gr^W_{m+1}\mathbb{H}^1(A_{\mathbb{Q}}^{\bullet}).
\end{aligned}
  \end{equation*}

\noindent Thus 

\begin{equation*}
    \begin{aligned}
Gr^W_{2}\mathbb{H}^1(A_{\mathbb{Q}}^{\bullet})=& \text{Cohomology of}\  \{E_1^{-2,2}\rightarrow {E}_1^{-1,2}\rightarrow {E}_1^{0,2}\}\\
=& \text{Cohomology of } \{\mathbb{H}^0(Gr^{M}_2A_{\mathbb{Q}}^{\bullet})\rightarrow \mathbb{H}^1(Gr^{M}_{1}A_{\mathbb{Q}}^{\bullet})\rightarrow \mathbb{H}^2(Gr^{M}_{0}A_{\mathbb{Q}}^{\bullet})\}.
    \end{aligned}
  \end{equation*}
    
\noindent By theorem \ref{2.2.3}, $$\mathbb{H}^0(Gr^{M}_2A_{\mathbb{Q}}^{\bullet})=0,\  \mathbb{H}^1(Gr^{M}_{1}A_{\mathbb{Q}}^{\bullet})=\mathbb{H}^1(Gr^W_2K^{\bullet}(1)[1])=H^0(X_0[2], \mathbb{Q}(1)),$$ and 
$$\mathbb{H}^2(Gr^{M}_{0}A_{\mathbb{Q}}^{\bullet})=\mathbb{H}^2(Gr^W_1K^{\bullet}(1)[1])=H^2(X_0[1], \mathbb{Q}(1)).$$ Thus $$L\cong\ker\{H^0(X_0[2], \mathbb{Q}(1))\rightarrow H^2(X_0[1], \mathbb{Q}(1))\}.$$
\end{proof}

Upto isogeny, we can take the free abelian group $L$ in the associated 1-motive $\mu_t: L\rightarrow G$ of $\mathbb{H}^1(A^{\bullet})$ to be $$L=\ker\{H^0(X_0[2], \mathbb{Z}(1))\rightarrow H^2(X_0[1], \mathbb{Z}(1))\}.$$

\begin{lemma}[{\cite[(4.27)]{st1}}]\label{4.1.2}
The sub-$\mathbb{Q}$-mixed Hodge structure $W_1\mathbb{H}^1(A^{\bullet})$ is isomorphic to the canonical mixed Hodge structure $H^1(X_0)$ of the singular curve $X_0$
\end{lemma}

\begin{proof} Let's sketch the idea of the proof on $\C$-level here, since we will use it in the later computation. Recall from Steenbrink \cite[(4.22)]{st1}, $\mathbb{H}^1(X_0, A^{\bullet})$ has an additional structure, the residue of Gauss-Manin connection $N$=Res$_0\nabla: \mathbb{H}^1(X_0, A^{\bullet})\rightarrow \mathbb{H}^1(X_0, A^{\bullet})(1)$, where ``$(1)$'' is the Tate twist of a mixed Hodge structure.

In fact, $N$ is induced by an endomorphism ${v}: A^{\bullet,\bullet}\rightarrow A^{\bullet+1, \bullet-1}$, the canonical projection $$\Omega_{X}^{p+q+1}(\log  X_0)/W_p\Omega_{X}^{p+q+1}(\log X_0)\rightarrow \Omega_{X}^{p+q+1}(\log  X_0)/W_{p+1}\Omega_{X}^{p+q+1}(\log X_0).$$

It is easy to see that\  $\ker v=$
\[
\xymatrix{
&\Omega^1_{X}\wedge\Omega_{X}^1(\log X_0)/\Omega_{X}^2\\
&\Omega_{X}^{1}(\log X_0)/\Omega_{X}^1\ar[u]_{d''}\ar[r]^-{d'}  &\Omega_{X}^2(\log X_0)/\Omega_{X}^1\wedge\Omega_{X}^1(\log X_0). \\
}
\tag{4.1}\label{diag 4.1}
\]

Through Poincar{\'e} residue map, $\ker v$ is isomorphic to
\[
\xymatrix{
&(a_1)_*\Omega^1_{X_0[1]}\\
&(a_1)_*\mathcal{O}_{X_0[1]}\ar[u]_{d}\ar[r]^-{\theta}  &(a_2)_*\underset{p\in X_0[2]}{\bigoplus}\mathbb{C}_{p}, \\
}
\tag{4.2}\label{diag 4.2}
\]
where $\theta$ is taking the difference of functions at $p$, according to the order of components of $X_0[1]$, i.e., $f
_i-f_j$ for $i<j$,\  where $f_i$ and $f_j$ are local functions on $X_{0,i}$ and $X_{0,j}$, respectively, and $p\in a_1(X_{0,i})\cap a_1(X_{0,j}).$

Note that the double complex (\ref{diag 4.2}) gives the mixed Hodge structure of $X_0$ and the double complex $\ker v$ gives the mixed Hodge structure $W_1\mathbb{H}^1(A^{\bullet})$. The same argument works for $\mathbb{Q}$-structures. Hence $H^1(X_0)\cong W_1\mathbb{H}^1(A^{\bullet})$ as $\mathbb{Q}$-mixed Hodge structures.
\end{proof}

\begin{lemma}\label{4.1.3}
$\mathbb{H}^1(A^{\bullet})/F^1\mathbb{H}^1(A^{\bullet})\cong W_1\mathbb{H}^1(A^{\bullet})/F^1W_1\mathbb{H}^1(A^{\bullet})$.
\end{lemma}
\begin{proof} This follows immediately from Deligne splitting of mixed Hodge structures \cite[lemma 3.4]{pest}. In fact,  for the mixed Hodge structure $(H, W_{\bullet}, F^{\bullet})$ of type $\{(0,0),(0,1),(1,0),(1,1)\}$, $W_1$ and $F^1$ generate $H$ as vector spaces.
\end{proof}

\begin{lemma}\label{4.1.4}
$W_1\mathbb{H}^1(A^{\bullet})/(F^1W_1\mathbb{H}^1(A^{\bullet})+W_1\mathbb{H}^1(A_{\mathbb{Q}}^{\bullet}))$ \textit{is isomorphic to }$Pic^0(X_0)\otimes\mathbb{Q}$.
\end{lemma}

\begin{proof} By lemma \ref{4.1.2}, we have
\begin{equation*}
    \begin{aligned}
 & W_1\mathbb{H}^1(A^{\bullet})/(F^1W_1\mathbb{H}^1(A^{\bullet})+W_1\mathbb{H}^1(A_{\mathbb{Q}}^{\bullet}))\\ 
\cong & H^1(X_0, \mathbb{C})/(F^1H^1(X_0, \mathbb{C})+H^1(X_0, \mathbb{Q}))\\
\cong & Pic^0(X_0)\otimes\mathbb{Q}
\end{aligned}
  \end{equation*}
\end{proof}
Upto isogeny, we can take the semi-abelian variety $G$ in the associated 1-motive $\mu_t: L\rightarrow G$ of $\mathbb{H}^1(A^{\bullet})$ to be $Pic^0(X_0)$.



Before we state next theorem, let's give the following ingredients.

\RN{1})\  The morphism \[
\xymatrix{
&\varphi: \mathbb{H}^1(A_{\mathbb{Q}}^{\bullet})\ar[r]&H^1(\mathcal{O}_{X_0})
}
\]
is the composition of canonical morphisms of cohomologies:
\[
\xymatrix{
\mathbb{H}^1(A_{\mathbb{Q}}^{\bullet})\ar[r]&\mathbb{H}^1(A^{\bullet})\ar[r]&\mathbb{H}^1(A^{\bullet})/F^1\mathbb{H}^1(A^{\bullet})=\mathbb{H}^1(Gr_F^0A^{\bullet})}
\]
\[
\xymatrix{
\cong W_1\mathbb{H}^1(A^{\bullet})/F^1W_1\mathbb{H}^1(A^{\bullet})\cong H^1(X_0, \mathbb{C})/F^1H^1(X_0, \mathbb{C})\cong H^1(\mathcal{O}_{X_0}),
}
\]
where $\mathbb{H}^1(A^{\bullet})/F^1\mathbb{H}^1(A^{\bullet})=\mathbb{H}^1(Gr_F^0A^{\bullet})$ is from the degeneration of Hodge spectral sequence at $E_1$ page. 
$$\mathbb{H}^1(Gr_F^0A^{\bullet})\cong W_1\mathbb{H}^1(A^{\bullet})/F^1W_1\mathbb{H}^1(A^{\bullet})$$ follows from lemma \ref{4.1.3}. 
$$W_1\mathbb{H}^1(A^{\bullet})/F^1W_1\mathbb{H}^1(A^{\bullet})\cong H^1(X_0, \mathbb{C})/F^1H^1(X_0, \mathbb{C})$$ follows from lemma \ref{4.1.2}, and $$H^1(X_0, \mathbb{C})/F^1H^1(X_0, \mathbb{C})=\mathbb{H}^1((a_1)_*\mathcal{O}_{X_0[1]}\overset{\theta}{\rightarrow} (a_2)_*\underset{p\in X_0[2]}{\bigoplus}\mathbb{C}_{p})$$ as that in the diagram (\ref{diag 4.2}).

Thus $\mathbb{H}^1(A_{\mathbb{Q}}^{\bullet})\rightarrow H^1(\mathcal{O}_{X_0})$ is induced by morphisms of complexes $$A_{\mathbb{Q}}^{\bullet}\rightarrow A^{\bullet}\rightarrow Gr_F^0A^{\bullet}\overset{P.R}{\longrightarrow}\ \{(a_1)_*\mathcal{O}_{X_0[1]}\overset{\theta}{\rightarrow} (a_2)_*\underset{p\in X_0[2]}{\bigoplus}\mathbb{C}_{p}\},$$ where $P.R$ is the Poincar\'e residue map as is explained in the proof of lemma \ref{4.1.2}.

\RN{2})\  We have canonical morphisms \[
\xymatrix{
&\mathbb{H}^1(A_{\mathbb{Q}}^{\bullet})\ar[r]& Gr_2^W\mathbb{H}^1(A_{\mathbb{Q}}^{\bullet})\ar@{^{(}->}[r]&\mathbb{H}^1(Gr_1^MA_\mathbb{Q}^{\bullet}),\\
}
\]
and the following lemma.

\begin{lemma}\label{4.1.5}
The composition of two maps\[
\xymatrix{
&\mathbb{H}^1(A_{\mathbb{Q}}^{\bullet})\ar@{->>}[r]& Gr_2^W\mathbb{H}^1(A_{\mathbb{Q}}^{\bullet})\ar@{^{(}->}[r]&\mathbb{H}^1(Gr_1^MA_\mathbb{Q}^{\bullet})\\
}
\] is induced by the projection $A_{\mathbb{Q}}^{\bullet}=M_1A_{\mathbb{Q}}^{\bullet}\rightarrow Gr_1^MA_\mathbb{Q}^{\bullet}$.
\end{lemma}

\begin{proof} Consider the monodromy weight spectral sequence
\begin{equation*}
E_1^{-m,1+m}=\mathbb{H}^1(Gr^{M}_mA_{\mathbb{Q}}^{\bullet})\Rightarrow \mathbb{H}^1(A_{\mathbb{Q}}^{\bullet})
\end{equation*}
induced by the filtered complex $M_{-1}A_{\mathbb{Q}}^{\bullet}\subset M_0A_{\mathbb{Q}}^{\bullet}\subset M_{1}A_{\mathbb{Q}}^{\bullet}$. It degenerates at $E_2$ page.\  $$E_{\infty}^{-1,2}=im\{\mathbb{H}^1(M_1A_{\mathbb{Q}}^{\bullet})\rightarrow \mathbb{H}^1(Gr_1^MA_\mathbb{Q}^{\bullet})\}.$$ Also, $$E_{\infty}^{-1,2}=E_2^{-1,2}=Gr_2^W\mathbb{H}^1(A_{\mathbb{Q}}^{\bullet}).$$

Also, by direct computation $M_1A_{\mathbb{Q}}^{\bullet,\bullet}=A_{\mathbb{Q}}^{\bullet,\bullet}=$
\[
\xymatrix{
0\\
\frac{\wedge^2\mathcal{M}}{\wedge^2\textbf{e}(\mathcal{O}_{X_0})}(1)\ar[u]\ar[r]&0\\
\frac{\mathcal{O}_{X_0}\otimes\mathcal{M}}{\mathcal{O}_{X_0}\otimes\textbf{e}(\mathcal{O}_{X_0})}(1)\ar[u]^d\ar[r]^-{\wedge 2\pi\sqrt[]{-1}\tilde{t}}&\frac{\wedge^2\mathcal{M}}{\textbf{e}(\mathcal{O}_{X_0})\wedge\mathcal{M}}(2)\ar[u]\ar[r] &0
}
\]
Thus $M_1A_{\mathbb{Q}}^{\bullet}=A_{\mathbb{Q}}^{\bullet}$.

\end{proof}

\RN{3})\  The composite morphism \[
\xymatrix{
&H^1(\mathcal{O}_{X_0})\ar[r]^-{exp}&Pic^0(X_0)\otimes\mathbb{Q}\ar@{^{(}->}[r]&Pic(X_0)\otimes\mathbb{Q}
}
\] 
is induced by the morphism of complexes
\[
\xymatrix{
&0\ar[r]&(a_1)_*\mathcal{O}_{X_0[1]}\ar[d]^{exp(2\pi i\centerdot)}\ar[r]^-{\theta}  &(a_2)_*\underset{p\in X_0[2]}{\bigoplus}\mathbb{C}_{p}\ar[r]\ar[d]^{exp(2\pi i\centerdot)}&0\\
&1\ar[r]&(a_1)_*\mathcal{O}^*_{X_0[1]}\otimes_{\mathbb{Z}}\mathbb{Q}\ar[r]^{\theta^*}&(a_2)_*\underset{p\in X_0[2]}{\bigoplus}\mathbb{C}^*_{p}\otimes_{\mathbb{Z}}\mathbb{Q}\ar[r]&1,
}
\]
where $\theta^*$ is taking the ratio of functions at $p$, according to the order of components of $X_0[1]$, i.e., $f
_i/f_j$ for $i<j$,\  where $f_i$ and $f_j$ are local functions on $X_{0,i}$ and $X_{0,j}$, respectively, and $p\in a_1(X_{0,i})\cap a_1(X_{0,j}).$

\begin{thm} \label{4.1.6}
The 1-motive map $\mu_t$ of the extension of\  $\mathbb{Q}$-mixed Hodge structures $0\rightarrow W_1\mathbb{H}^1(A^{\bullet})\rightarrow \mathbb{H}^1(A^{\bullet})\rightarrow Gr^W_2\mathbb{H}^1(A^{\bullet})\rightarrow 0$ is the morphism making the following diagram commutative
\[
\xymatrix{
&H^1(\mathcal{O}_{X_0})\ar[r]^-{exp}&Pic^0(X_0)\otimes\mathbb{Q}\\
&\mathbb{H}^1(A_{\mathbb{Q}}^{\bullet})\ar[r]^-{pr}\ar[u]^{\varphi}&L\otimes\Q\ar@{-->}[u]_{\mu_t}\ar@{=}[r]& Gr_2^W\mathbb{H}^1(A_{\mathbb{Q}}^{\bullet}),\\
}
\tag{4.3}\label{diag 4.3}
\]
where $pr: \mathbb{H}^1(A_{\mathbb{Q}}^{\bullet})\rightarrow L\otimes\mathbb{Q}$ is the usual projection map. The morphisms in the diagram (\ref{diag 4.3}), except $\mu_t$, are discussed in \RN{1}), \RN{2}), \RN{3}) previously.
\end{thm}

\begin{proof}
Note that we have\  $\ker(exp)=H^1(X_0, \mathbb{Q})$. Also, $\ker(pr)=W_{1}\mathbb{H}^1(A_{\mathbb{Q}}^{\bullet})$.  Then by lemma \ref{4.1.2} and the construction of $\varphi$ we have the following commutative diagram:
\[
\xymatrix{
0\ar[r]&H^1(X_0, \mathbb{Q})\ar[r]&H^1(\mathcal{O}_{X_0})\ar@{->>}[r]^-{exp}&Pic^0(X_0)\otimes\mathbb{Q}\ar[r]&0\\
0\ar[r]&W_{1}\mathbb{H}^1(A_{\mathbb{Q}}^{\bullet})\ar[u]^{\cong}\ar[r]&\mathbb{H}^1(A_{\mathbb{Q}}^{\bullet})\ar@{->>}[r]^-{pr}\ar[u]^{\varphi}&L\otimes\mathbb{Q}\ar[r]&0.\\
}
\]
Thus there exist a unique map $\mu_t: L\otimes\mathbb{Q}\rightarrow Pic^0(X_0)\otimes\mathbb{Q}$ making the above diagram commutative. Then by construction (2) of section \ref{abstract 1-motive}, $\mu_t$ is the 1-motive up to isogeny of the extension $0\rightarrow W_1\mathbb{H}^1(A^{\bullet})\rightarrow \mathbb{H}^1(A^{\bullet})\rightarrow Gr^W_2\mathbb{H}^1(A^{\bullet})\rightarrow 0$.

What's more, by diagram chasing, $\mathbb{H}^1(A_{\mathbb{Q}}^{\bullet})$ is the fiber product of morphisms $exp$ and $\mu_t$.
\end{proof}

\subsection{Geometric description of 1-motives $\RN{1}$} \label{geo 1}

In the last three subsections, we want to find out a nice geometric description of the morphism $\mu_t: L\rightarrow Pic^0(X_0)$ up to isogeny, where $\mu_t$ is described in theorem \ref{4.1.6}.\\ 

Note that we have the following short exact sequence:
\[
\xymatrix{
1\ar[r]&{\mathbb{C}^*}^{d+1-n}\ar[r]&Pic^0(X_0)\ar[r]^{a_1^*}&Pic^0(X_0[1])\ar[r]&1,\\
}
\]
where $d$ is the number of nodal points of $X_0$, and $n$ is the number of irreducible components of $X_0$. Geometrically, any line bundle $\mathcal{L}$ over $X_0$ is coming from the pullback line bundle $\tilde{\mathcal{L}}:=a_1^*\mathcal{L}$ over $X_0[1]$ glued along the preimage of nodal points. The gluing data is encoded in ${\mathbb{C}^*}^{d+1-n}$.\ \ Based on the above discussion, to understand $\mu_t$, we first calculate $a_1^*\circ \mu_t$ up to isogeny in the following commutative diagram
\[
\xymatrix{
&H^1(\mathcal{O}_{X_0})\ar[r]^-{exp}&Pic^0(X_0)\otimes\mathbb{Q}\ar@{->>}[r]^-{a_1^*}&Pic^0(X_0[1])\otimes\mathbb{Q}\\
&\mathbb{H}^1(A_{\mathbb{Q}}^{\bullet})\ar[r]^-{pr}\ar[u]^{\varphi}& Gr_2^W\mathbb{H}^1(A_{\mathbb{Q}}^{\bullet})\ar[u]^{\mu_t}.\\
}
\tag{4.4}\label{diag 4.4}
\]

Before we formulate next theorem, let's make the following notation.

\begin{note}\label{4.2.1}
 For any nodal point $p\in X_0[2]$, we denote $a_1^{-1}(p)$ by $\{p',p''\}$, where $p'$ is contained in $X_{0,i}$ and $p''$  is contained in $X_{0,j}$, for $i<j$.
\end{note}

\begin{thm}\label{4.2.2}
The morphism $$a_1^*\circ \mu_t: \ker  \{H_0(X_0[2], \mathbb{Z})\rightarrow H_0(X_0[1], \mathbb{Z})\}\otimes \mathbb{Q} \rightarrow Pic(X_0[1]\otimes\mathbb{Q})$$ is given by $$D=\underset{p\in X_0[2]}{\Sigma}n_p p\mapsto \mathcal{O}_{X_0[1]}(\underset{p\in X_0[2]}{\Sigma}n_p (p'-p'')),$$ where $a_1^*D$ is a degree zero divisor on each irreducible component of $X_0[1].$
\end{thm}

Before the proof of theorem \ref{4.2.2}, we prove the following lemma first. Denote the set of nodal points of $X_0$ by $S$. Consider the normalization map $a_1: X_0[1]\rightarrow X_0$. We have the open smooth curve $X_0[1]-a_1^{-1}(S)$, which we denote by $Y$.

\begin{lemma}\label{4.2.3}
There is a canonical injective morphism of $\mathbb{Q}$-mixed Hodge structures $\mathbb{H}^1(A^{\bullet})/W_0\mathbb{H}^1(A^{\bullet})\hookrightarrow H^1(Y)$.
\end{lemma}

\begin{proof} In \cite[section 5.3]{hain}, we have another construction of the nearby fiber $X_{\infty}$ through real oriented blowup. Let $X'_0$ be the real oriented blowup of $X_0[1]$ along $a_1^{-1}(S)$. Then $X_{\infty}$ is obtained from $X'_0$ by gluing, the gluing data for which is given by the local defining equation of $X_0$ and the parameter $t$ we chose. Topologically, we have $$Y\hookrightarrow X'_0\rightarrow X_{\infty}.$$

The induced map $\mathbb{H}^1(A^{\bullet})\cong H^1(X_{\infty})\rightarrow H^1(Y)$ is a morphism of $\mathbb{Q}$-mixed Hodge structures by \cite[theorem 14]{hain}. Then we have the following commutative diagram of $\mathbb{Q}$-mixed Hodge structures
\[
\xymatrix{
&&0\ar[d]\\
0\ar[r]&W_0H^1(X_0)\ar[d]^{\cong}\ar[r]&H^1(X_0)\ar[r]\ar[d]&H^1(Y)\ar[r]& Gr^W_2H^1(Y)\ar[r]&0\\
0\ar[r]&W_0\mathbb{H}^1(A^{\bullet})\ar[r]&\mathbb{H}^1(A^{\bullet})\ar[d]\ar[ru]\\
&&Gr_2^W\mathbb{H}^1(A^{\bullet})\ar@{^{(}->}[rruu]\ar[d].\\
&&0
}
\tag{4.5}\label{diag 4.5}
\]

The first row is an exact sequence of $\mathbb{Q}$-mixed Hodge structures, which is obtained from the inclusion of curves $Y\hookrightarrow X_0[1]$. The isomorphism in the first column is obtained from lemma \ref{4.1.2}. Also, by lemma \ref{4.1.2}, the second column is a short exact sequence of $\mathbb{Q}$-mixed Hodge structures. The inclusion map $Gr_2^WH^1(A^{\bullet})\hookrightarrow Gr^W_2H^1(Y)$ is obtained from lemma \ref{4.1.1}. Thus, from the above diagram, we get an exact sequence of $\mathbb{Q}$-mixed Hodge structures 
$$0\rightarrow W_0\mathbb{H}^1(A^{\bullet})\rightarrow \mathbb{H}^1(A^{\bullet})\rightarrow H^1(Y).$$
\end{proof}

\begin{proof}[Proof of theorem \ref{4.2.2}]
Consider the $\mathbb{Q}$-mixed Hodge structure for the open curve $Y$. By construction (2) in section \ref{abstract 1-motive}, we have the following commutative diagram
\[
\xymatrix{
H^1(\mathcal{O}_{X_0[1]})\ar[r]^-{exp}&Pic^0(X_0[1])\otimes\mathbb{Q}\\
H^1(Y, \mathbb{Q})\ar[r]^-{pr}\ar[u]& Gr_2^WH^1(Y, \mathbb{Q})\ar[u]^{\mu_Y}\\
}
\tag{4.6}\label{diag 4.6}
\]
where $\mu_Y$ is the 1-motive associated to the $\mathbb{Q}$-mixed Hodge structure $H^1(Y)$. Then by lemma \ref{4.2.3}, we have the following commutative diagram
\[
\xymatrix{
H^1(\mathcal{O}_{X_0})\ar[r]&H^1(\mathcal{O}_{X_0[1]})\ar[r]^-{exp}&Pic^0(X_0[1])\otimes\mathbb{Q}\\
\mathbb{H}^1(A_{\mathbb{Q}}^{\bullet})\ar[r]\ar[u]^{\varphi}&\mathbb{H}^1(A_{\mathbb{Q}}^{\bullet})/W_0\mathbb{H}^1(A_{\mathbb{Q}}^{\bullet})\ar[r]^-{pr}\ar[u]& Gr_2^W\mathbb{H}^1(A_{\mathbb{Q}}^{\bullet})\ar[u]_{\mu_Y|_{Gr_2^W\mathbb{H}^1(A_{\mathbb{Q}}^{\bullet})}}.\\
}
\tag{4.7}
\]

Thus we have $a_1^*\circ \mu_t\cong\mu_Y|_{Gr_2^W\mathbb{H}^1(A_{\mathbb{Q}}^{\bullet})}$. By Deligne \cite[(10.3.8)]{de1}, the 1-motive map $$\mu_Y|_{Gr_2^W\mathbb{H}^1(A_{\mathbb{Q}}^{\bullet})}: \ker  \{H_0(X_0[2], \mathbb{Z})\rightarrow H_0(X_0[1], \mathbb{Z})\}\otimes \mathbb{Q} \rightarrow Pic(X_0[1]\otimes\mathbb{Q})$$ is just the map $$D=\underset{p\in X_0[2]}{\Sigma}n_p p\mapsto \mathcal{O}_{X_0[1]}(\underset{p\in X_0[2]}{\Sigma}n_p (p'-p'')),$$ where $a_1^*D$ is a degree zero divisor on each component of $X_0[1].$ Hence theorem \ref{4.2.2} holds.
\end{proof}

\subsection{Geometric description of 1-motives $\RN{2}$} \label{geo 2}

In subsection \ref{geo 1}, we have already known that if $D=\underset{p\in X_0[2]}{\Sigma} n_p p\in L$, the 1-motive $\mu_t(D)$ can be obtained from a line bundle $\mathcal{O}_{X_0[1]}(\underset{p\in X_0[2]}{\Sigma} n_p (p'-p''))$ with gluing data along the pairs $(p',p'')$.\ In this subsection, let's describle an educated guess of 1-motives $\mu_t$ associated to the limit mixed Hodge structure of a degeneration of curves in Deligne \cite{de2}.

Recall that in the introduction, we have $f:X\rightarrow \Delta$, with a chosen parameter $t: \Delta\rightarrow \mathbb{C}$. It gives a global section $\tilde{t}:=t\circ f\in\Gamma(X, \mathcal{M}_X)$, where $\mathcal{M}_X=\mathcal{O}_X\cap j_*\mathcal{O}^*_U$. $\tilde{t}$ has a zero along each component $\{X_{0,i}\}$ of order 1, since the monodromy is unipotent.

For any nodal point $p\in X_0\subset X$, we choose an open neighborhood $V_p$ of $p$ in analytic topology of $X$, such that $f:V_p\rightarrow\Delta$ can be defined to be $\{(u,v,\tilde{t})\in \mathbb{C}^3|\ u\cdot v=\tilde{t}\}\overset{pr_3}{\longrightarrow}\C$, i.e., the following diagram commutes:
\[
\xymatrix{
&&W\ar@{^{(}->}[d]^-{\text{closed immersion}}\\
V_p\ar[rru]\ar[rd]_-{\tilde{t}}\ar[rr]&&\mathbb{C}^2\times\mathbb{C}\ar[ld]^-{pr_2}\\
&\mathbb{C}
}
\tag{4.8}\label{diag 4.8}
\]
where $W$ is $\{(u,v,\tilde{t})\in \mathbb{C}^3|\ u\cdot v=\tilde{t}\}$. Note that $\{u=0\}$ and $\{v=0\}$ give the two components of $V_p\cap X_{0}$. Thus $u|_{\{v=0\}}$ can be regarded as a local coordinate function at $p'$  and $v|_{\{u=0\}}$ can be regarded as a coordinate function at $p''$ in terms of notation \ref{4.2.1}.

\begin{note} \label{4.3.1}
Pick a set of open complex polydisks $\{V_{\alpha}\}_{\alpha\in I}$ in $X$, s.t., $\mathcal{U}=\{U_{\alpha}:=V_{\alpha}\cap X_0\}_{\alpha\in I}$ is a covering of $X_0$. We also assume that $V_{\alpha}, U_{\alpha}, V_{\alpha}\cap V_{\beta}$ and $U_{\alpha}\cap U_{\beta}$ are simply connected, connected and containing at most one nodal point, for all $\alpha, \beta$. Also assume that every nodal point is contained in at most one $U_{\alpha}, V_{\alpha}$.
\end{note}

For the covering\  $\mathcal{U}=\underset{\alpha\in I}{\bigcup}U_{\alpha}$ of $X_0$\ in the above notation, as is described in the diagram (\ref{diag 4.8}), we can choose $u_{\alpha},\  v_{\alpha}$ for any $V_{\alpha}$ containing some nodal point $p$ in $X_0$.\  We always denote $u_{\alpha}$ to be the coordinate function at $p'$ and $v_{\alpha}$ to be the coordinate function at $p''$, in terms of notation \ref{4.2.1}.\\

\begin{defin} \label{4.3.2}
Define $$\nu_t: L=\ker\{H_0(X_0[2], \mathbb{Z})\rightarrow H_0(X_0[1], \mathbb{Z})\} \longrightarrow Pic^0(X_0)$$ as follows.

For any element $D=\underset{p\in X_0[2]}{\Sigma} n_p p\in L$, attach it with the line bundle\  $$\mathcal{L}(D):=\mathcal{O}_{X_0[1]}(\underset{p\in X_0[2]}{\Sigma} n_p (p'-p''))\in Pic^0(X_0[1]).$$ For each point $p\in U_{\alpha}$,\  consider the following gluing data:

\[
\xymatrix{
\rho_p: &\text{Fiber of}\  \mathcal{L}(D)\  \text{at}\  p'\ar[r]^-{\sim}&\text{Fiber of}\  \mathcal{L}(D)\  \text{at}\  p''\\
&\mathcal{L}(D)\otimes(\mathcal{O}_{X_0[1],p'}/m_{p'})\ar@{=}[u]\ar[r]^-{\sim}&\mathcal{L}(D)\otimes(\mathcal{O}_{X_0[1],p''}/m_{p''})\ar@{=}[u],
}
\tag{4.9}\label{diag 4.9}
\]
induced by the commutative diagram:
\[
\xymatrix{
&\mathcal{L}(D)\otimes(\mathcal{O}_{X_0[1],p'}/m_{p'})\ar[d]_-{(u_{\alpha})^{n_p}}^{\cong}\ar[r]^-{\overset{\rho_p}{\sim}}&\mathcal{L}(D)\otimes(\mathcal{O}_{X_0[1],p''}/m_{p''})\ar[d]^-{(1/v_{\alpha})^{n_p}}_{\cong}\\
&\mathcal{O}_{X_0[1],p'}/m_{p'}\ar[d]&\mathcal{O}_{X_0[1],p''}/m_{p''}\ar[d]\\
&\mathbb{C}\ar@{=}[r]&\mathbb{C}.\\
}
\tag{4.10}\label{diag 4.10}
\]
Gluing the line bundle $\mathcal{L}(D)$ according to the above gluing data, we get a line bundle in  $Pic^0(X_0)$, which is defined to be the image $\nu_t(D)$.
\end{defin}

Next we want to describe the $\nu_t(D)$
in terms of $\check{\text{C}}$ech double complex. 

We already have the resolution $$1\rightarrow \mathcal{O}^*_{X_0}\overset{a_1^*}{\rightarrow} (a_1)_*\mathcal{O}^*_{X_0[1]}\overset{\theta^*}{\rightarrow} (a_2)_*\underset{p\in X_0[2]}{\bigoplus}\mathbb{C}^*_{p}\rightarrow 1$$ of $\mathcal{O}^*_{X_0}$.\  The $\check{\text{C}}$ech double complex of $(a_1)_*\mathcal{O}^*_{X_0[1]}\overset{\theta^*}{\rightarrow} (a_2)_*\underset{p\in X_0[2]}{\bigoplus}\mathbb{C}^*_{p}$ is
\[
\xymatrix{
&1 \\
&C^0(\mathcal{U}, (a_2)_*\underset{p\in X_0[2]}{\bigoplus}\mathbb{C}^*_{p})\ar[u]\ar[r]&1\\
(0,0)\ar@{~>}[r]&C^0(\mathcal{U}, (a_1)_*\mathcal{O}^*_{X_0[1]})\ar[u]^-{\theta^*}\ar[r]^-{\delta}&C^1(\mathcal{U}, (a_1)_*\mathcal{O}^*_{X_0[1]})\ar[u]^-{\theta^*}\ar[r]^-{\delta} &\cdots\\
&1\ar[u]&1\ar[u] \\
}
\tag{4.11}\label{diag 4.11}
\]

Denote the above $\check{\text{C}}$ech double complex by $E^{\bullet, \bullet}$ and its total complex by $E^{\bullet}$. Then we have the canonical isomorphisms\[\check{\text{H}}^1(E^{\bullet})\cong \mathbb{H}^1((a_1)_*\mathcal{O}^*_{X_0[1]}\overset{\theta^*}{\rightarrow} (a_2)_*\underset{p\in X_0[2]}{\bigoplus}\mathbb{C}^*_{p})\cong H^1(X_0, \mathcal{O}^*_{X_0})\cong Pic(X_0),\]
so we get a canonical morphism $$\Psi:  C^0(\mathcal{U}, (a_2)_*\underset{p\in X_0[2]}{\bigoplus}\mathbb{C}^*_{p})\bigoplus  Z^1(\mathcal{U}, (a_1)_*\mathcal{O}^*_{X_0[1]})\rightarrow Pic(X_0),$$ where $Z^1(\mathcal{U}, (a_1)_*\mathcal{O}^*_{X_0[1]})$ is the group of 1-cocycles. It is also clear that $$C^0(\mathcal{U}, (a_2)_*\underset{p\in X_0[2]}{\bigoplus}\mathbb{C}^*_{p})=\underset{p\in X_0[2]}{\bigoplus}\mathbb{C}^*_{p}.$$ We want to find a representative of $\nu_t(D)$ in $(\underset{p\in X_0[2]}{\bigoplus}\mathbb{C}^*_{p})\bigoplus  Z^1(\mathcal{U}, (a_1)_*\mathcal{O}^*_{X_0[1]})$, for $D=\underset{p\in X_0[2]}{\Sigma} n_p p\in L$.

\begin{note}\label{4.3.3}
For any node $p\in U_{\alpha}$, where $U_{\alpha}\in\mathcal{U}$, we have local coordinate functions $u_{\alpha}, v_{\alpha}$ near $p'$ and $p''$ as is discussed previously. We denote the function $w^{\alpha}$ over $a_1^{-1}U_{\alpha}$ to be $u_{\alpha}$ over the connected component containing $p'$, and $1/v_{\alpha}$ over the other component containing $p''$.  If $U_{\beta}$ does not contain nodal point, we denote $w^{\beta}=1$.
\end{note}

\begin{thm}\label{4.3.4}
For any $D=\underset{p\in X_0[2]}{\Sigma}n_pp\in L$, let
\begin{equation*}
    n_{\alpha}=
    \begin{cases}
  n_p & \text{for}\ p\in U_{\alpha}\\
  1 &  \text{for} \ U_{\alpha}\  \text{containing no nodal point}
  \end{cases}
\end{equation*}
Then $$((1,...,1), \{{\frac{(\omega^{\alpha})^{n_\alpha}}{(\omega^{\beta})^{n_{\beta}}}}|_{U_{\alpha\beta}}\})\in (\underset{p\in X_0[2]}{\bigoplus}\mathbb{C}^*_{p})\bigoplus  Z^1(\mathcal{U}, (a_1)_*\mathcal{O}^*_{X_0[1]})$$ represents $\nu_t(D)$ in $Pic^0(X_0)$, where $U_{\alpha\beta}=U_{\alpha}\cap U_{\beta}$
\end{thm}

\begin{proof} Note that the quasi-isomorphism of complexes
\[
\xymatrix{
1\ar[r]&\mathcal{O}^*_{X_0}\ar[r]\ar[d]_-{a_1^*}&1\ar[d]\\
1\ar[r]&(a_1)_*\mathcal{O}^*_{X_0[1]}\ar[r]^-{\theta^*}&(a_2)_*\underset{p\in X_0[2]}{\bigoplus}\mathbb{C}^*_{p}\ar[r] &1\\
}
\]
gives the morphism of two $\check{\text{C}}$ech double complexes as below:
\[
\xymatrix{
&1\ar[r]&1\\
&C^0(\mathcal{U}, (a_1)_*\mathcal{O}^*_{X_0})\ar@{=>}[d]\ar[u]^-{\theta^*}\ar[r]^-{\delta}&C^1(\mathcal{U}, (a_1)_*\mathcal{O}^*_{X_0})\ar@{=>}[d]\ar[u]^-{\theta^*}\ar[r]^-{\delta} &\cdots\\
&C^0(\mathcal{U}, (a_2)_*\underset{p\in X_0[2]}{\bigoplus}\mathbb{C}^*_{p})\ar[r]&1\\
&C^0(\mathcal{U}, (a_1)_*\mathcal{O}^*_{X_0[1]})\ar[u]^-{\theta^*}\ar[r]^-{\delta}&C^1(\mathcal{U}, (a_1)_*\mathcal{O}^*_{X_0[1]})\ar[u]^-{\theta^*}\ar[r]^-{\delta} &\cdots\\
}
\]
The $\check{\text{C}}$ech 1-cocycle $$\{{\frac{(\omega^{\alpha})^{n_\alpha}}{(\omega^{\beta})^{n_{\beta}}}}|_{U_{\alpha\beta}}\}\in Z^1(\mathcal{U}, (a_1)_*\mathcal{O}^*_{X_0})$$ maps to the representative $((1,...,1), \{{\frac{(\omega^{\alpha})^{n_\alpha}}{(\omega^{\beta})^{n_{\beta}}}}|_{U_{\alpha\beta}}\})$ through the above diagram. Also, the equality in the bottom row of the diagram (\ref{diag 4.10}) corresponds to $(1,...,1)$, and it is clear that the representative $$\{{\frac{(\omega^{\alpha})^{n_\alpha}}{(\omega^{\beta})^{n_{\beta}}}}|_{U_{\alpha\beta}}\}\in Z^1(\mathcal{U}, (a_1)_*\mathcal{O}^*_{X_0})$$ gives the line bundle $\nu_t(D)\in Pic^0(X_0)$. Thus $$((1,...,1), \{{\frac{(\omega^{\alpha})^{n_\alpha}}{(\omega^{\beta})^{n_{\beta}}}}|_{U_{\alpha\beta}}\})\in (\underset{p\in X_0[2]}{\bigoplus}\mathbb{C}^*_{p})\bigoplus  Z^1(\mathcal{U}, (a_1)_*\mathcal{O}^*_{X_0[1]})$$ represents $\nu_t(D)$.
\end{proof}

\subsection{Geometric description of 1-motives $\RN{3}$} \label{geo 3}

Recall that our ultimate goal is to describe the geometric meaning of the morphism $\mu_t: L\rightarrow Pic^0(X_0)$ upto isogeny. In this subsection, we will prove that the 1-motive $\mu_t$ coincides with $\nu_t$ up to isogeny, where $\nu_t$ is defined in subsection \ref{geo 2}. In order to achieve this goal, let's consider the rational structure $(A_{\mathbb{Q}}^{\bullet}, M_{\bullet})$ given in section \ref{double complex}.

We have double complexes $A_{\mathbb{Q}}^{\bullet, \bullet}$, $M_0A_{\mathbb{Q}}^{\bullet, \bullet}$, $Gr_1^MA_{\mathbb{Q}}^{\bullet, \bullet}$ as follows:

$(1)\ A_{\mathbb{Q}}^{\bullet, \bullet}=M_1A_{\mathbb{Q}}^{\bullet, \bullet}=$

\[
\xymatrix{
0\\
\frac{\wedge^2\mathcal{M}}{\wedge^2\textbf{e}(\mathcal{O}_{X_0})}(1)\ar[u]\ar[r]&0\\
\frac{\mathcal{O}_{X_0}\otimes\mathcal{M}}{\mathcal{O}_{X_0}\otimes\textbf{e}(\mathcal{O}_{X_0})}(1)\ar[u]^d\ar[r]^-{\wedge 2\pi\sqrt[]{-1}\tilde{t}}&\frac{\wedge^2\mathcal{M}}{\textbf{e}(\mathcal{O}_{X_0})\wedge\mathcal{M}}(2)\ar[u]\ar[r] &0
}
\tag{4.12}\label{diag 4.12}
\]

$(2)\ M_0A_{\mathbb{Q}}^{\bullet, \bullet}=$
\[
\xymatrix{
0\\
\frac{\textbf{e}(\mathcal{O}_{X_0})\wedge\mathcal{M}}{\wedge^2\textbf{e}(\mathcal{O}_{X_0})}(1)\ar[u]\ar[r]&0\\
\frac{\mathcal{O}_{X_0}\otimes\mathcal{M}}{\mathcal{O}_{X_0}\otimes\textbf{e}(\mathcal{O}_{X_0})}(1)\ar[u]^d\ar[r]^-{\wedge 2\pi\sqrt[]{-1}\tilde{t}}&\frac{\wedge^2\mathcal{M}}{\textbf{e}(\mathcal{O}_{X_0})\wedge\mathcal{M}}(2)\ar[u]\ar[r] &0
}
\tag{4.13}\label{diag 4.13}
\]

$(3)\ Gr_1^MA_{\mathbb{Q}}^{\bullet, \bullet}=$

\[
\xymatrix{
0\\
\frac{\wedge^2\mathcal{M}}{\textbf{e}(\mathcal{O}_{X_0})\wedge\mathcal{M}}(1)\ar[u]\ar[r]^-{\theta}&0\\
0\ar[u]^d\ar[r]&0\ar[u]\\
}
\tag{4.14}\label{diag 4.14}
\]

\begin{lemma}\label{4.4.1}
$$\frac{\wedge^2\mathcal{M}}{\textbf{e}(\mathcal{O}_{X_0})\wedge\mathcal{M}}\cong(a_2)_*\mathbb{Q}_{X_0[2]}.$$


\end{lemma}

\begin{proof} (1) Consider the morphism of sheaves of free abelian groups $\textbf{e}: \mathcal{O}_{X_0}\rightarrow \mathcal{M}$. Denote $\textbf{f}: \mathcal{O}_{X_0}\rightarrow \textbf{e}(\mathcal{O}_{X_0})$. \  Note that $$\frac{\wedge^2\mathcal{M}}{\textbf{e}(\mathcal{O}_{X_0})\wedge\mathcal{M}}=Gr_2^{W(\textbf{e}(\mathcal{O}_{X_0}))}Kos^2(\textbf{e})^2.$$ By proposition \ref{1.2.4}, we have \begin{equation*}
    \begin{aligned}
  &  Gr_2^{W(\textbf{e}(\mathcal{O}_{X_0}))}Kos^2(\textbf{e})^2\\
    \cong & Kos^0(\textbf{f})^2[-2]\otimes \overset{2}{\bigwedge}(\coker \textbf{e})\\
    \cong & \overset{2}{\bigwedge}((a_1)_*\mathbb{Q}_{X_0[1]})\\
    \cong &(a_2)_*\mathbb{Q}_{X_0[2]}.
    \end{aligned}
    \end{equation*}
\end{proof}

Consider the diagram (\ref{diag 4.3})
\[
\xymatrix{
&H^1(\mathcal{O}_{X_0})\ar[r]^-{exp}&Pic^0(X_0)\otimes\mathbb{Q}\\
&\mathbb{H}^1(A_{\mathbb{Q}}^{\bullet})\ar[r]^-{pr}\ar[u]^{\varphi}&L\otimes\mathbb{Q}\ar[u]_{\mu_t}\ar@{=}[r]& Gr_2^W\mathbb{H}^1(A_{\mathbb{Q}}^{\bullet}).\\
}
\]

We want to work on the level of complexes, so we extend the above square to be the following commutative diagram:
\[
\xymatrix{
H^1(\mathcal{O}_{X_0})\ar[r]^-{exp}&Pic^0(X_0)\otimes\mathbb{Q}\ar@{^{(}->}[r]^{inclusion}&Pic(X_0)\otimes\mathbb{Q}\\
\mathbb{H}^1(A_{\mathbb{Q}}^{\bullet})\ar[r]^-{pr}\ar[u]^-{\varphi}& Gr_2^W\mathbb{H}^1(A_{\mathbb{Q}}^{\bullet})\ar[u]_{\mu_t}\ar@{^{(}->}[r]^{inclusion}&\mathbb{H}^1(Gr_1^MA_{\mathbb{Q}}^{\bullet})\\
&\ker\{H_0(X_0[2], \mathbb{Z})\rightarrow H_0(X_0[1], \mathbb{Z})\}\otimes \mathbb{Q}\ar@{=}[u]^{\cong}\ar@{^{(}->}[r]&H_0(X_0[2], \mathbb{Z})\otimes \mathbb{Q}\ar@{=}[u]^{\cong}.\\
}
\tag{4.15}\label{diag 4.15}
\]

To compute $\mathbb{H}^1(A_{\mathbb{Q}}^{\bullet}), \ \mathbb{H}^1(Gr_1^MA_{\mathbb{Q}}^{\bullet}), \ H^1(\mathcal{O}_{X_0})$, \ and $Pic(X_0)\otimes\mathbb{Q}$, we use  $\check{\text{C}}$ech cohomology. Taking in to account lemma \ref{4.4.1} and the choice of the covering $\mathcal{U}$ in notation \ref{4.3.1}, we have that in the  $\check{\text{C}}$ech complex of $\frac{\wedge^2\mathcal{M}}{\textbf{e}(\mathcal{O}_{X_0})\wedge\mathcal{M}}$, $$C^i(\mathcal{U}, \frac{\wedge^2\mathcal{M}}{\textbf{e}(\mathcal{O}_{X_0})\wedge\mathcal{M}})=0,$$ for $i>0$, since $\frac{\wedge^2\mathcal{M}}{\textbf{e}(\mathcal{O}_{X_0})\wedge\mathcal{M}}$ is a skyscraper sheaf by lemma \ref{4.4.1}. Thus we have the following $\check{\text{C}}$ech double complexes of $A_{\mathbb{Q}}^{\bullet}$, $Gr_1^MA_{\mathbb{Q}}^{\bullet}$, and  $(a_1)_*\mathcal{O}^*_{X_0[1]}\overset{\theta^*}{\rightarrow} (a_2)_*\underset{p\in X_0[2]}{\bigoplus}\mathbb{C}^*_{p}$ which is denoted to be $C^{\bullet, \bullet}, D^{\bullet, \bullet}$, and $E^{\bullet, \bullet}$, respectively. \\

(1)\ $C^{\bullet, \bullet}$
\[
\xymatrix{
0&0\\
C^0(\mathcal{U}, \frac{\wedge^2\mathcal{M}}{\wedge^2\textbf{e}(\mathcal{O}_{X_0})}(1)\bigoplus\frac{\wedge^2\mathcal{M}}{\textbf{e}(\mathcal{O}_{X_0})\wedge\mathcal{M}}(2))\ar[u]\ar[r]^-{\delta}&C^1(\mathcal{U}, \frac{\wedge^2\mathcal{M}}{\wedge^2\textbf{e}(\mathcal{O}_{X_0})}(1))\ar[u]\ar[r]^-{\delta} &\cdots\\
C^0(\mathcal{U}, \frac{\mathcal{O}_{X_0}\otimes\mathcal{M}}{\mathcal{O}_{X_0}\otimes\textbf{e}(\mathcal{O}_{X_0})}(1))\ar[u]^D\ar[r]^-{\delta}&C^1(\mathcal{U}, \frac{\mathcal{O}_{X_0}\otimes\mathcal{M}}{\mathcal{O}_{X_0}\otimes\textbf{e}(\mathcal{O}_{X_0})}(1))\ar[u]^D\ar[r]^-{\delta} &\cdots\\
}
\tag{4.16}\label{diag 4.16}
\]
where $D=(d, \wedge2\pi\sqrt[]{-1}\tilde{t})$

(2)\ $D^{\bullet, \bullet}$
\[
\xymatrix{
0\\
C^0(\mathcal{U}, \frac{\wedge^2\mathcal{M}}{\textbf{e}(\mathcal{O}_{X_0})\wedge\mathcal{M}}(1))\ar[u]\ar[r]&0\\
0\ar[u]\ar[r]&0\ar[u]\\
}
\tag{4.17}\label{diag 4.17}
\]\\

(3)\ $E^{\bullet, \bullet}$

\[
\xymatrix{
1 \\
C^0(\mathcal{U}, (a_2)_*\underset{p\in X_0[2]}{\bigoplus}\mathbb{C}^*_{p})\ar[u]\ar[r]&1\\
C^0(\mathcal{U}, (a_1)_*\mathcal{O}^*_{X_0[1]})\ar[u]^-{\theta^*}\ar[r]^-{\delta}&C^1(\mathcal{U}, (a_1)_*\mathcal{O}^*_{X_0[1]})\ar[u]\ar[r]^-{\delta} &\cdots\\
}
\tag{4.18}\label{diag 4.18}
\]

Denote their total complex to be $C^{\bullet}, D^{\bullet}$, and $E^{\bullet}$, respectively. Also, we have morphisms $\Psi_1: C^{\bullet, \bullet} \rightarrow D^{\bullet, \bullet}$ induced by natural projection,\ and  $\Psi_2: C^{\bullet, \bullet} \rightarrow E^{\bullet, \bullet}$ induced by canonical morphisms $$A_{\mathbb{Q}}^{\bullet, \bullet}\rightarrow [(a_1)_*\mathcal{O}_{X_0[1]}\overset{\theta}{\rightarrow}  (a_2)_*\underset{p\in X_0[2]}{\bigoplus}\mathbb{C}_{p}]\otimes_{\mathbb{Z}}\mathbb{Q}\overset{exp}{\rightarrow} [(a_1)_*\mathcal{O}^*_{X_0[1]}\overset{\theta^*}{\rightarrow} (a_2)_*\underset{p\in X_0[2]}{\bigoplus}\mathbb{C}^*_{p}]\otimes_{\mathbb{Z}}\mathbb{Q},$$ which is described in  \RN{1}) section \ref{abs}.\\

The following theorem is our main theorem.

\begin{thm}\label{4.4.2}
The abstract 1-motive $\mu_t$ associated to the $\Q$-limit mixed Hodge structure  $\mathbb{H}^1(A^{\bullet}$) coincides with the geometric 1-motive $\nu_t$ up to isogeny, where $\nu_t$ constructed in section \ref{geo 2}.
\end{thm}

\begin{proof}
Take any element $D=\underset{p\in X_0[2]}{\Sigma}n_pp\in L$. Consider the line bundle $$\mathcal{L}(D):=\mathcal{O}_{X_0[1]}(\underset{p\in X_0[2]}{\Sigma} n_p (p'-p''))\in Pic^0(X_0[1]),$$ where $a_1^{-1}(p)=\{p', p''\}$. 
We take local functions  $(\omega^{\alpha})^{n_\alpha}$ on $a_1^{-1}(U_{\alpha})$ of $\mathcal{L}(D)$ for each $\alpha$ as in theorem \ref{4.3.4}, where $w^{\alpha}$ is defined in notation \ref{4.3.3}.  Then we have the 1-cocycle$$\{{\frac{(\omega^{\alpha})^{n_\alpha}}{(\omega^{\beta})^{n_{\beta}}}}|_{U_{\alpha\beta}}\}\in Z^1(\mathcal{U}, (a_1)_*\mathcal{O}^*_{X_0[1]})$$ which represents the line bundle $\mathcal{L}(D)$,  where $U_{\alpha\beta}=U_{\alpha}\cap U_{\beta}$. Also, by theorem \ref{4.3.4}, we have the 1-cocycle $$((1,1,...,1), \{{\frac{(\omega^{\alpha})^{n_\alpha}}{(\omega^{\beta})^{n_{\beta}}}}|_{U_{\alpha\beta}}\})\in (\underset{p\in X_0[2]}{\bigoplus}\mathbb{C}^*_{p})\bigoplus  Z^1(\mathcal{U}, (a_1)_*\mathcal{O}^*_{X_0[1]})$$ represents $\nu_t(D)$.  For simplicity, we denote $(\omega^{\alpha})^{n_\alpha}$ by $g^{\alpha}$.

Next we want to check that through the morphism $\Psi_2$ from $C^{\bullet,\bullet}$ to $E^{\bullet,\bullet}$, we can lift the 1-cocycle $((1,1,...,1), \{{\frac{g^{\alpha}}{g^{\beta}}}|_{U_{\alpha\beta}}\})$ to a 1-cocycle in $$C^0(\mathcal{U}, \frac{\wedge^2\mathcal{M}}{\wedge^2\textbf{e}(\mathcal{O}_{X_0})}(1)\bigoplus\frac{\wedge^2\mathcal{M}}{\textbf{e}(\mathcal{O}_{X_0})\wedge\mathcal{M}}(2))\bigoplus  C^1(\mathcal{U}, \frac{\mathcal{O}_{X_0}\otimes\mathcal{M}}{\mathcal{O}_{X_0}\otimes\textbf{e}(\mathcal{O}_{X_0})}(1)).$$

Let $G^{\alpha}_{\beta}$ be a holomorphic extension of $g^{\alpha}|_{U_{\alpha\beta}}$ in $V_{\alpha\beta}$, where $V_{\alpha}$ is introduced in notation \ref{4.3.1} and $V_{\alpha\beta}=V_{\alpha}\cap V_{\beta}$.

(1) When $U_{\alpha}$ contains a node, depending on $\beta$, $G^\alpha_{\beta}$ can be extended to be a holomorphic function $(u_{\alpha})^{n_{\alpha}}$ or a meromorphic function $(1/v_{\alpha})^{n_{\alpha}}$ in $V_{\alpha}$. \ Also note that since $\mathcal{M}$ is groupification of a multiplicative monoid, thus in $\mathcal{M}$ we have $$(u_{\alpha})^{n_{\alpha}}\wedge \tilde{t}=((u_{\alpha})^{n_{\alpha}}/\tilde{t}^{n_{\alpha}})\wedge \tilde{t}=(1/v_{\alpha})^{n_{\alpha}}\wedge\tilde{t}.$$ Thus we have a well-defined extension $G^{\alpha}_{\beta}\wedge\tilde{t}$ over ${V_{\alpha}}$, which is independent of $\beta$. Hence we can denote the extension by $G^{\alpha}\wedge\tilde{t}$.

(2) When $U_{\alpha}$ does not contain any node, $G^{\alpha}_{\beta}$ can be extended to some holomorphic function $G^{\alpha}$ over $V_{\alpha}$. Then we also have a well-defined element $G^{\alpha}\wedge\tilde{t}$ over ${V_{\alpha}}$.

Then we get $\{2\pi\sqrt[]{-1}\overline{G^{\alpha}\wedge \tilde{t}}\}\in C^0(\mathcal{U}, \frac{\wedge^2\mathcal{M}}{\textbf{e}(\mathcal{O}_{X_0})\wedge\mathcal{M}}(1))$, which represents the element $D=\underset{p\in X_0[2]}{\Sigma}n_pp\in L$ through the $\check{\text{C}}$ech double complex $D^{\bullet,\bullet}$.  \ In the  $\check{\text{C}}$ech double complex $C^{\bullet,\bullet}$,

we consider element $$\kappa=(\{2\pi\sqrt[]{-1}\overline{G^{\alpha}\wedge \tilde{t}}\},\ \{(2\pi\sqrt[]{-1})^2\overline{G^{\alpha}\wedge \tilde{t}}\},\ 2\pi\sqrt[]{-1}(\log G^{\alpha}_{\beta}-\log G^{\beta}_{\alpha})\otimes \tilde{t})$$

$$\in C^0(\mathcal{U}, \frac{\wedge^2\mathcal{M}}{\wedge^2\textbf{e}(\mathcal{O}_{X_0})}(1)\bigoplus\frac{\wedge^2\mathcal{M}}{\textbf{e}(\mathcal{O}_{X_0})\wedge\mathcal{M}}(2))\bigoplus  C^1(\mathcal{U}, \frac{\mathcal{O}_{X_0}\otimes\mathcal{M}}{\mathcal{O}_{X_0}\otimes\textbf{e}(\mathcal{O}_{X_0})}(1)).$$
Now we check that it is actually a 1-cocycle. Note first that in the diagram (\ref{diag 4.16}), since $d(\log G^{\alpha}_{\beta}-\log G^{\beta}_{\alpha})=G^{\alpha}_{\beta}-G^{\beta}_{\alpha}$, where $d$ is the differential of the kozul complex. we have 
\begin{equation*}
\begin{aligned}
& \delta(\{2\pi\sqrt[]{-1}\overline{G^{\alpha}\wedge \tilde{t}}\},\ \{(2\pi\sqrt[]{-1})^2\overline{G^{\alpha}\wedge \tilde{t}}\})\\
= & \delta(\{2\pi\sqrt[]{-1}\overline{G^{\alpha}\wedge \tilde{t}}\})\\
= & D(2\pi\sqrt[]{-1}(\log G^{\alpha}_{\beta}-\log G^{\beta}_{\alpha})\otimes \tilde{t})
\end{aligned}
\end{equation*}
in diagram (\ref{diag 4.16}), where the first equality holds because 
$C^i(\mathcal{U}, \frac{\wedge^2\mathcal{M}}{\textbf{e}(\mathcal{O}_{X_0})\wedge\mathcal{M}}(1))=0$, for $i>0$. Also, note that $\delta(2\pi\sqrt[]{-1}(\log G^{\alpha}_{\beta}-\log G^{\beta}_{\alpha})\otimes \tilde{t})=0$.  Thus $\kappa$ is a 1-cocycle.

Through morphism $\Psi_1$, the image of $\kappa$ is $\{2\pi\sqrt[]{-1}G^{\alpha}\wedge \tilde{t}\}$, which represents $D=\underset{p\in X_0[2]}{\Sigma}n_pp\in L$.  Through morphism $\Psi_2$, the image of $\kappa$ is $((1,1,...,1), \{{\frac{g^{\alpha}}{g^{\beta}}}|_{U_{\alpha\beta}}\})$, which represents $\nu_t(D)$ by theorem \ref{4.3.4}. Therefore we proved that $\mu_t$ coincides with $\nu_t$ up to isogeny.
\end{proof}

\begin{rmk}
The above proof can only detemine the 1-motive map $\mu_t$ up to isogeny, since we didn't specify the integral structure for the limit mixed Hodge structure. As mentioned in the \cite{hoff}, it is pretty interesting to find a canonical object that resolves the complex of nearby cycles in the derived category $D^+(X_0, \mathbb{Z})$, which will give the $\mathbb{Z}$-limit mixed Hodge structure from geometry. Then our proof can be applied to find the actual 1-motive corresponding to the $\mathbb{Z}$-limit mixed Hodge structure.

Also, from the above theorem, we can see that the 1-motive $\mu_t$ only depends on the first order deformation of the central fiber $X_0$. 

\end{rmk}

\end{document}